\documentclass[11pt,letterpaper]{amsart}

\usepackage{mathrsfs,color,balance,bm,amsmath,amsfonts,amssymb,dsfont,amscd,extarrows,enumerate,verbatim}
\usepackage[all]{xy}
\numberwithin{equation}{section}

\newtheorem{thm}{Theorem}[section]
\newtheorem{cor}[thm]{Corollary}
\newtheorem{lem}[thm]{Lemma}
\newtheorem{remark}[thm]{Remark}
\newtheorem{prop}[thm]{Proposition}

\newcommand{\mr}{\mathbb{R}}

\newcommand{\mc}{\mathbb{C}}

\newcommand{\rw}{\rightarrow}

\DeclareMathOperator{\diam}{diam}

\makeatletter
\def\@settitle{\begin{center}%
  \baselineskip14\p@\relax
  \bfseries
  \uppercasenonmath\@title
  \@title
  \ifx\@subtitle\@empty\else
     \\[1ex]\uppercasenonmath\@subtitle
     \footnotesize\mdseries\@subtitle
  \fi
  \end{center}%
}
\def\subtitle#1{\gdef\@subtitle{#1}}
\def\@subtitle{}
\makeatother

\usepackage[pagebackref=true,colorlinks=true,bookmarksopen,bookmarksnumbered,citecolor=red, linkcolor=blue, urlcolor=cyan]{hyperref}

\newcommand{\clubtag}{\hyperref[ineq:club]{\ensuremath{\clubsuit}}}

\allowdisplaybreaks[3]
\begin{document}

\title[Estimates of heat kernels and Sobolev-type inequalities]
{Estimates of heat kernels and Sobolev-type inequalities for twisted differential forms on compact K\"ahler manifolds}

\author[F. Deng]{Fusheng Deng}
\address{Fusheng Deng: \ School of Mathematical Sciences, University of Chinese Academy of Sciences \\ Beijing 100049, P. R. China}
\email{fshdeng@ucas.ac.cn}

\author[G. Huang]{Gang Huang }
\address{Gang Huang: \ School of Mathematical Sciences, University of Chinese Academy of Sciences \\ Beijing 100049, P. R. China}
\email{huanggang21@mails.ucas.ac.cn}

\author[X. Qin]{Xiangsen Qin}
\address{Xiangsen Qin: \ School of Mathematical Sciences, University of Chinese Academy of Sciences \\ Beijing 100049, P. R. China}
\email{qinxiangsen19@mails.ucas.ac.cn}
\maketitle
\begin{abstract}
    The main goal of this paper is to   generalize  the Sobolev-type inequalities given by Guo-Phong-Song-Sturm and Guedj-Tô 
from the case of functions to the framework of twisted differential forms.
To this end, we establish certain estimates of heat kernels for differential forms with values in holomorphic vector bundles
      over compact K\"ahler manifolds. As applications of these estimates,
      we also prove a vanishing theorem and give certain $L^{q,p}$-estimates for the $\bar\partial$-operator on  twisted differential forms. 
    %Furthermore, we get an improved $L^2$-estimate of H\"ormander on a strictly pseudoconvex open subset of a K\"ahler manifold.
\end{abstract}
%\thanks{(*) The second author and the third author are both corresponding authors.}

\maketitle
\tableofcontents
\section{Introduction} 
   In the remarkable recent  works \cite{GPSS24} and \cite{GT24} (see also \cite{GPSS22}, \cite{GPSS23}, and \cite{GPS24}),
   the authors establish some Sobolev-type inequalities for functions associated with a family of K\"ahler metrics satisfying certain entropy bounds on compact K\"ahler manifolds. In the present paper, we try to generalize these inequalities to the framework of differential forms with values in Hermitian holomorphic vector bundles.
    However, in attempting to  follow the approaches in \cite{GPSS24} and \cite{GT24}, one will inevitably encounter  the problem of how to define the Monge-Amp\'ere operator for
     differential forms or sections of a holomorphic vector bundle.  
     Our strategy to circumvent this difficulty is trying to produce a heat kernel estimate for twisted differential forms 
     from a Sobolev-type inequality for functions.\\
    \par  To state the  result, we first introduce the context we are studying.
  Let $M$ be a compact K\"ahler manifold of complex dimension $n$, whose volume is denoted by $|M|$.
  Let $E$ be a Hermitian holomorphic vector bundle of rank $r$ over $M$, $\square_{p,q}$ be the $\bar\partial$-Laplacian which acts on  $C^2(M,\Lambda^{p,q} T^*M\otimes E)$ for some fixed $p,q\in\{0,\cdots,n\},$ and let $\nabla $ be the Chern connection on $\Lambda^{p,q}T^*M\otimes E$.
 Then the Weitzenb\"ock curvature operator is given by 
 $$\mathfrak{Ric}_{p,q}^E:=2\square_{p,q}-\nabla^*\nabla.$$
 We denote by $H_{p,q}(t,x,y)$ the heat kernel associated to $\square_{p,q}$.
 Let $b_{p,q}:=\dim \mathcal{H}_{p,q}(M,E)$ be the dimension of the space $\mathcal{H}_{p,q}(M,E)$ of harmonic $(p,q)$-forms on $M$ with values in $E$. 
 %and $\phi_1,\cdots,\phi_{b_{p,q}}$ be an orthonormal basis of $\mathcal{H}_{p,q}(M,E)$.
 Then the first main result can be stated as follows:
 \begin{thm}\label{thm: heat kernel}
  With the above notations, and assume that there are constants $\alpha>1,\ \beta>0,\gamma\geq 0$
    such that the following Sobolev-type inequality for functions is satisfied:
\begin{equation}\label{ineq:club}
\left( \int_{M} |f|^{2\alpha} \right)^{\frac{1}{\alpha}} \leq\beta\int_{M} |\nabla f|^{2} +\gamma\int_{M} |f|^2,\ \forall f\in C^\infty(M)
\tag{\clubtag}.
\end{equation}
We also assume $\mathfrak{Ric}_{p,q}^E\geq -K$ for some constant $K\geq 0$.
Then the following estimates hold.    
   \begin{itemize}
     \item[(i)] There is a constant $C:=C(\alpha)>0$ such that 
     $$|H_{p,q}(t,x,y)|\leq C\beta^{\frac{\alpha}{\alpha-1}}e^{Kt+\frac{\gamma}{\beta}t}t^{\frac{\alpha}{1-\alpha}}e^{-\frac{d(x,y)^2}{9t}}$$
     for all $(t,x,y)\in\mr_+\times M\times M.$
     \item[(ii)] There is a constant $C:=C(\alpha)>0$ such that 
      
      $$\left|H_{p,q}(t,x,y)-\sum_{k=1}^{b_{p,q}}\phi_k(x)\otimes \phi_k(y)\right|\leq C
      r\beta^{\frac{2\alpha}{\alpha-1}}|M|e^{\left(\frac{K}{\mu_1}+\frac{\gamma}{\beta\mu_1}\right)
   \frac{2\alpha}{\alpha-1}}t^{\frac{2\alpha}{1-\alpha}}$$
   for all $(t,x,y)\in\mr_+\times M\times M,$
      where $\mu_1$ is the first nonzero eigenvalue of $\square_{p,q}$, and $\phi_1,\cdots,\phi_{b_{p,q}}$ is an orthonormal basis of $\mathcal{H}_{p,q}(M,E)$.
     \item[(iii)]  If $\mathfrak{Ric}_{p,q+1}^E\geq -K_+,\ \mathfrak{Ric}_{p,q-1}^E\geq -K_{-}$ for some constants $K_+,K_{-}\geq 0$.
     Then  
     \begin{align*}
     &\quad \max\{|\bar\partial_y H_{p,q}(t,x,y)|,|\bar\partial_y^* H_{p,q}(t,x,y)|\}\\
     &  \leq r\beta^{\frac{2\alpha}{\alpha-1}}|M|e^{\left(\frac{K+K_{+}+K_{-}}{\mu_1}+\frac{\gamma}{\beta\mu_1}\right)
   \frac{2\alpha}{\alpha-1}}t^{\frac{2\alpha}{1-\alpha}-\frac{1}{2}}
   \end{align*}
   for all $(t,x,y)\in\mr_+\times M\times M$.
   \item[(iv)] There is a constant $C:=C(\alpha,\beta,\gamma,K,K_+,K_{-})>0$ such that 
   $$\max\{|\bar\partial_y H_{p,q}(t,x,y)|^2,|\bar\partial_y^* H_{p,q}(t,x,y)|^2\}\leq Cn|B(y,\sqrt{t})|t^{\frac{4\alpha-1}{1-\alpha}}e^{-\frac{d(x,y)^2}{18t}}$$
    for all $x,y\in M,\ 0<t\leq 1$, where $B(y,\sqrt{t})$ is the geodesic ball in $M$ with center  $y$ and radius $\sqrt{t}$.
   \end{itemize}
  \end{thm}
 \indent When $M$ is a real manifold and $E$ is trivial, $C^0$- and $C^1$-estimates for the heat kernel similar to ours have been obtained in \cite{WZ12} or \cite{BDG23}. However, how their constants depend on the geometric quantities of $M$ differs from that of ours.  For $(p,q)=(n,0)$ and $E$ is a trivial line bundle, \cite{LZZ21} considered $C^0$-estimates and covariant derivative estimates
with a exponential term depending on the bounds of the Ricci curvature of $M$. For general complete K\"ahler manifolds, $L^p$-estimates for the heat kernels of twisted forms under the assumption $\mathfrak{Ric}_{p,q}^E>\rho$ were given in 
 in \cite{LX10}, where $\rho$ is a positive constant. Following the same argument of the proof for Part (i) of Theorem \ref{thm: heat kernel},
 one can see that (i) of Theorem \ref{thm: heat kernel} remains valid if 
 $M$ is replaced by a general complete Kähler manifold satisfying the Sobolev-type inequality (\clubtag).
 \\

 %Note that in 
 %Theorem \ref{thm: heat kernel}, we don't provide estimates for covariant derivatives of the heat kernels
 %as this would introduce additional  curvature terms not controlled by our assumptions (since $\square_{p,q}\nabla\neq \nabla\square_{p,q}$).\\ 
 \indent We now outline the main ideas behind the proof of  Theorem  \ref{thm: heat kernel}.
 Utilizing ideas from  \cite{D90} and \cite{SC02} (see also \cite{Z11}), one can show that the Sobolev-type inequality (\clubtag) imply a Gaussian-type upper bound
 for the heat kernel of functions, then yielding Part (i) by \cite[Theorem 4.3]{LX10}. The proof of Part (ii) follows the methods in \cite{WZ12} and \cite{LZZ21}, involving  estimates for  the eigenvalues and eigensections of $\square_{p,q}$.
 The proof of (iii) is motivated by the proof of \cite[Theorem 3.1]{LZZ21} and relies crucially on certain $C^1$-estimates for solutions to the heat equation associated with $\square_{p,q}$, established in Lemma \ref{lem:  the solution of heat equation}.\\
 \indent Next we present several consequences of  Theorem \ref{thm: heat kernel}.\\
 \indent The relationship between the heat kernel and the Green form (the Schwarz kernel of the Green operator for $\square_{p,q}$) leads to the following estimates for the Green form $G_{p,q}:$ 
 \begin{cor}\label{cor:green form}
 Under the same assumptions and notations of Theorem \ref{thm: heat kernel}.  Let $G_{p,q}$ be the Schwarz kernel of the Green operator for $\square_{p,q}$,
   %acts on $C^2(\overline\Omega,\Lambda^pT^*M\otimes E)$,
   then we have the following estimates:   
   \begin{itemize}
     \item[(i)]  There is a constant $C:=C(\alpha,\beta,\gamma,\mu_1,K,|M|)>0$ such that 
     $$|G_{p,q}(x,y)|\leq Cd(x,y)^{\frac{2}{1-\alpha}}+C,\ \forall (x,y)\in M\times M.$$
     \item[(ii)] Suppose the Ricci curvature of $M$ is bounded below by $-(2n-1)\theta$ for some constant $\theta\geq 0$,
     and  assume $\alpha$ satisfies 
     $$\frac{4\alpha-1}{\alpha-1}-n-2> 0,$$
      then there is a constant 
      $$C:=C(\alpha,\beta,\gamma,r,n,\mu_1,\theta,K,K_+,K_{-},\diam(M))>0$$
      such that for all $x,y\in M$, we have 
      $$ \max\left\{|\bar\partial_yG_{p,q}(x,y)|,|\bar\partial^{*}_yG_{p,q}(x,y)|\right\}\leq Cd(x,y)^{\frac{2\alpha+1}{1-\alpha}+n}.$$
   \end{itemize} 
     \end{cor}
  \indent When $M$ is a real manifold and $E$ is a trivial line bundle, some $L^p$-estimates of Green forms and their derivatives were given in \cite{SC95},
  but the constants remain unspecified.
  By \cite{D90} (see also \cite{L12} and \cite{Z11}), diagonal heat kernel estimates for functions imply Sobolev-type inequalities for functions.
  Employing similar ideas, we derive the following Sobolev-type inequality for forms from Theorem \ref{thm: heat kernel}.
   This inequality serves as our key tool for constructing Sobolev-type inequalities for twisted differential forms 
   on compact K\"ahler manifolds within the framework of \cite{GPSS23} and 
  \cite{GT24}.
  %By \cite[Theorem 1.2]{RM10}, the following Sobolev-type inequality is a consequence of (ii) of .
 \begin{cor}\label{cor:good sobolev}
 Under the same assumptions and notations of Theorem \ref{thm: heat kernel}. 
 For any $1\leq \ell<(2\alpha)/(\alpha-1)$, set 
 $$-m:=\frac{1}{2}+\frac{\alpha}{(1-\alpha)\ell},$$
 and suppose $1\leq k<\infty$ satisfies
 $$
 \begin{cases}
                             k<\ell/(2m)+\ell, & \mbox{if } \ell=1, \\
                             k\leq \ell/(2m)+\ell, & \mbox{if }  1<\ell<(2\alpha)/(\alpha-1). \\
    \end{cases}
 $$ 
 Then there is a constant  $C:=C(k,\ell,\alpha,\beta,\gamma,r,\mu_1,K,|M|)>0$
 such that the following inequality holds:
 $$\left(\int_{M}\left|f-\mathcal{P}f\right|^{k}dV\right)^{1/k}
 \leq C\left(\int_{M}|\square_{p,q}f|^{\ell}dV\right)^{1/\ell}$$
for all 
 $f\in C^2(M,\Lambda^{p,q}T^*M\otimes E)$, and where  $\mathcal{P}\colon L^2(M,\Lambda^{p,q}T^*M\otimes E)\rw \mathcal{H}_{p,q}(M,E)$ denotes the Bergman projection.\\
 \indent In particular, there is a constant $\delta:=\delta(\alpha,\beta,\gamma,r,\mu_1,K,|M|)>0$ such that for all 
 $f\in C^2(M,\Lambda^{p,q}T^*M\otimes E)$, we have 
 $$\delta\left(\int_{M}\left|f-\mathcal{P}f\right|^{2\alpha}dV\right)^{1/\alpha}\leq\int_{M}|\square_{p,q}f|^2=\int_{M}\left(|\bar\partial f|^2+|\bar\partial^* f|^2\right)dV.
 $$
 \end{cor}
 \indent By Corollary \ref{cor:Bergman}, the Bergman projection $\mathcal{P}$ is $L^k$-bounded for any $1\leq k\leq \infty$, so 
 under the assumptions of Corollary \ref{cor:good sobolev},  we have
 $$\left(\int_{M}\left|f\right|^{2\alpha}dV\right)^{1/\alpha}
 \leq C\left(\int_{M}|\bar\partial f|^2dV+\int_{M}|\bar\partial^* f|^2dV+\int_{M}|f|^2dV\right).$$
 \indent When $n\geq 2$, it is well known that the Sobolev-type inequality (\clubtag) holds with $\alpha:=n/(n-1)$.
 One need to note that when $\alpha=n/(n-1)$, the restrictions relation on $k,\ell$ naturally arise in the classical Sobolev-type inequality.
  Employing ideas from \cite{DHQ25} (see also \cite[Lemma 2.10, 2.11]{DHQ24} and their proofs) and Corollary \ref{cor:green form}, one can construct the following two Sobolev-type inequalities:
 \begin{cor}\label{cor:weak sobolev}
 Under the same assumptions and notations of Theorem  \ref{thm: heat kernel}. Assume further $n\geq 2$
 and the Ricci curvature of $M$ is bounded below by $-(2n-1)\theta$ for some constant $\theta\geq 0$. 
 Let $1\leq k,\ell\leq \infty$ satisfy 
 $$  \begin{cases}
                             k<\frac{n\ell}{n-\ell}, & \mbox{if } \ell=1,n, \\
                             k\leq \frac{n\ell}{n-\ell}, & \mbox{if }  1<\ell<n. \\
    \end{cases}
 $$
 Then there is a constant  $C:=C(k,\ell,r,n,\mu_1,\theta,K,\diam(M))>0$
 such that the following inequality holds:
 $$\left(\int_{M}\left|f-\mathcal{P}f\right|^{k}\right)^{\frac{1}{k}}
 \leq C\left(\int_{M}\left|\square_{p,q}f\right|^{\ell}\right)^{\frac{1}{\ell}},\ \forall f\in C^2(M,\Lambda^{p,q}T^*M\otimes E).$$
 \end{cor}
 \indent We remark in Corollary \ref{cor:weak sobolev}, it is possible that $\ell\geq n$. One need to note that the constants in Corollary \ref{cor:good sobolev} and  \ref{cor:weak sobolev} exhibit different dependencies
 on the parameters.
 \begin{cor}\label{cor:dbar estimate}
 Under the same assumptions and notations of Theorem  \ref{thm: heat kernel}. Assume further $n\geq 2$
 and the Ricci curvature of $M$ is bounded below by $-(2n-1)\theta$ for some constant $\theta\geq 0.$
 Suppose $1\leq k,s,t\leq \infty$ satisfy 
 $$  \begin{cases}
                             k<\frac{2ns}{2n-s}, & \mbox{if } s=1,2n, \\
                             k\leq \frac{2ns}{2n-s}, & \mbox{if }  1<s<2n, \\
    \end{cases}
    \quad 
  \begin{cases}
                             k<\frac{2nt}{2n-t}, & \mbox{if } t=1,2n, \\
                             k\leq \frac{2nt}{2n-t}, & \mbox{if }  1<t<2n. \\
    \end{cases}
 $$
 Then there is a constant  $\delta:=\delta(k,s,t,r,n,\mu_1,\theta,K,K_+,K_{-},\diam(M))>0$
 such that for all $f\in C^2(M,\Lambda^{p,q}T^*M\otimes E)$,  the following inequality holds:
 $$\delta\left(\int_{M}\left|f-\mathcal{P}f\right|^{k}dV\right)^{\frac{1}{k}}
 \leq \left(\int_{M}\left|\bar\partial f\right|^{s}dV\right)^{\frac{1}{s}}
 +\left(\int_{M}\left|\bar\partial^* f\right|^{t}dV\right)^{\frac{1}{t}}.$$
 \end{cor}
 \indent The above two corollaries also have the corresponding versions when $n=1$
 since in this case, the Sobolev-type inequality (\clubtag) holds for any $2<\alpha<\infty$.
 We omit the precise statements here.\\
  \indent Using Hodge theory, we derive from  Corollary \ref{cor:green form} the following vanishing theorem and $L^{k,s}$-estimates for the $\bar\partial$-operator.
 
 \begin{thm}\label{thm:complex L^p estimate}
    Under the same assumptions and notations of Corollary \ref{cor:dbar estimate} and we further assume 
    $s>1$, $\mathfrak{Ric}_{p,q}^E\geq 0$ and $\mathfrak{Ric}_{p,q}^E>0$ at some point.
     Then there is a constant 
    $$C:=C(k,s,r,n,\mu_1,\theta,K_{-},\diam(M))>0$$ 
    such that 
    for any $\bar\partial$-closed $f\in L^s(M,\Lambda^{p,q}T^*M\otimes E)$, we may find a $u\in L^k(M,\Lambda^{p,q-1}T^*M\otimes E)$ satisfying
    $$\bar\partial u=f,\ \left(\int_{M}\left|u\right|^{k}dV\right)^{\frac{1}{k}}
 \leq C\left(\int_{M}\left|f\right|^{s}dV\right)^{\frac{1}{s}}.$$ 
  Furthermore, when $s=2$, the constant  $C$ can be chosen to depend only on 
  $k,r,n,\mu_1,|M|$ and the Sobolev constants $\beta,\gamma$ when $\alpha=n/(n-1).$
 \end{thm}
 \indent When $M$ is a real manifold, $L^{k,s}$-estimates for $d$ were constructed in  \cite{GT06}
 by using $L^{k,s}$-cohomology. For the case $k=s$, Theorem \ref{thm:complex L^p estimate} was established in \cite{LX10} under stronger positivity assumptions,
 yielding a better constant.  For compact K\"ahler manifold with boundaries, similar $L^{k,s}$-estimates of $\bar\partial $ have been constructed in \cite{DHQ25}.\\
  \indent We remark that the estimates in  the above results involving the first nonzero eigenvalue $\mu_1$ of $\square_{p,q}$.
  In general, we don't have any meaningful  estimate of it so far. On the other hand,
  under some curvature positivity assumptions, we can give a lower bound of $\mu_1$ as follows. 
   \begin{thm}\label{thm: first eigenvalue}
 Let $M$ be a compact K\"ahler manifold of complex dimension $n$.
  Fix $p,q\in\{0,\cdots,n\}.$ Let $E$ be a Hermitian holomorphic vector bundle over $M$  such that there is a real valued  function 
  $b$ over $M$ satisfying 
  $$\mathfrak{Ric}_{p,q}^E\geq b\geq 0,$$
  and $b(x)>0$ at some point $x\in M$.
  Then we have 
 $$\mu_1\geq \frac{\mu_1^0}{4|M|}\int_{M}\frac{b}{\mu_1^0+b},$$
 where $\mu_1^0$ is the first nonzero eigenvalue of $\Delta$.
 \end{thm}
  %if $\mathfrak{Ric}_{p,q}^E\geq 0$ and the inequality is strict at some point of $M$, 
 %then $\mathcal{H}_{p,q}(M,E)=0$ and  we can obtain a lower 
 %bound of $\mu_1$ in terms of the first nonzero eigenvalue $\mu_1^0$  of the Laplace-Beltrami operator $\Delta$  acting on functions, see  Corollary \ref{cor:vanish}. This yields  an $L^2$-estimate for solutions to the  $\bar\partial$-equation,
 %see Corollary \ref{cor:L2 estimate}. \\

 \par With the above preparations, we now generalize  the Sobolev-type inequalities in \cite[Theorem 2.1]{GPSS23} and \cite[Therorem 2.6]{GT24} 
from the case of functions to the framework of twisted differential forms.
 We focus on compact K\"ahler manifolds within the framework of
 \cite{GPSS23}, as the other case can be treated similarly.  
% \indent With Theorem \ref{thm: heat kernel} in hand, we generalize \cite[Theorem 2.1]{GPSS23} 
 %to the following family version of Sobolev-type inequalities:
 \begin{thm}\label{thm:family version}
 Let $E$ be a Hermitian holomorphic vector bundle over a compact K\"ahler manifold $(X,\omega_0)$ of complex dimension $n$.
 Given constants $m>n,\ A>0,\ B>0,\ K\geq 0,$ and  $p,q\in\{0,\cdots,n\}$, there are constants $\alpha:=\alpha(n,m)>1$ and 
 $C:=C(X,\omega_0,n,m,\alpha,A,B,K)>0$ such that for any  K\"ahler metric $\omega\in \mathcal{V}_{p,q}^E(X, \omega_0, n,m,A,B,K)$ and any $f\in C^1(X,
 \Lambda^{p,q}T^*X\otimes E)$, the following Sobolev-type inequality holds:
$$\left(\int_{X}\left|f-\mathcal{P}_\omega f\right|^{2\alpha}\omega^n\right)^{1/\alpha}\leq CV_\omega^{\frac{1}{\alpha}-1}
(1+\eta)^2\left(\int_{X}|\bar\partial f|^{2}\omega^n+\int_{X}|\bar\partial^* f|^{2}\omega^n\right),$$
where 
$$\eta:=2r(I_\omega K+1)^{\frac{2\alpha}{\alpha-1}}+e^K
   +I_\omega^{\frac{2\alpha}{\alpha-1}}e^{K+I_\omega^{-1}}+rI_\omega^{\frac{2\alpha}{\alpha-1}}e^{\frac{2(KI_\omega+1)\alpha}{(\alpha-1)I_\omega\mu_{1,\omega}}}.$$
\indent Furthermore, if $K=0,$  then there is a constant $c:=c(X,\omega_0,n,m,A,B)>0$
such that for any $\omega\in \mathcal{V}_{p,q}^E(X, \omega_0, n,m,A,B,0)$ satisfying $b_{p,q,\omega}\geq 1$, we have 
$$\mu_{1,\omega}\geq \frac{c}{I_\omega}.$$
 \end{thm}
 \indent For the various notations and a more general version of Theorem \ref{thm:family version}, please see Section \ref{sec:family}. When $E$ is a trivial line bundle, and $(p,q)=(0,0)$, we have $\mathfrak{Ric}_{0,0,\omega}^E=0,\ b_{0,0,\omega}=1$ for any 
 K\"ahler metric $\omega$. Thus, Theorem \ref{thm:family version} generalizes \cite[Theorem 2.1]{GPSS23}.\\
% Using ideas from \cite{GPSS23}, it is also possible to extend Theorem \ref{thm:family version} to compact K\"ahler spaces
 %with singularities, but we don't pursue this here.\\
 \begin{comment}
 %By using the same ideas, one may also generalize \cite[Theorem ]{GT24} to the case of twisted differential forms.\\
 \indent In the last, we will give a family version of  Theorem \ref{thm:complex L^p estimate}.
 \begin{thm}\label{thm:L^p estimates family}
 Under the setting of \cite{GPSS23} or \cite{GT24}, and we moreover assume $\mathfrak{Ric}_{p,q,\omega}^E\geq -K$ for some constant $K\geq 0$
 where $q\geq 1$.
 Suppose  $1<k,s<\infty$ satisfying 
 $$k(2n-s)\leq 2ns.$$
 Then there is a constant 
 $C>0$ which only depends on the initial metric, the entropy bounds and $\alpha, K,k,s$ such that for any $\bar\partial$-closed $f\in L^s(M,\Lambda^{p,q}T^*M\otimes E)$, we may find  a $u\in L^k(M,\Lambda^{p,q-1}T^*M\otimes E)$ such that 
    $$\bar\partial u=f,\ \left(\int_{M}\left|u\right|^{k}dV_\omega\right)^{\frac{1}{k}}
 \leq C\left(\int_{M}\left|f\right|^{s}dV_\omega\right)^{\frac{1}{s}}.$$ 
 \end{thm}
 \indent Clearly, similar conclusions of Theorem \ref{thm:L^p estimates family} also hold under the framework of \cite{GT24}.\\
 \end{comment}
 \indent This paper is organized as follows. In Section \ref{notations}, we fix some notations and conventions that are needed in our discussions.  In Section \ref{sec:main theorem}, 
 we will give some estimates of eigenvalues, eigensections, and heat kernels.
 Particularly, we will give the proof of  Theorem \ref{thm: heat kernel},
 the proof of Theorem \ref{thm: first eigenvalue} and a vanishing theorem of $\mathcal{H}_{p,q}(M,E)$
 (Corollary \ref{cor:vanish}), and an $L^2$-estimate for $\bar\partial$ (Corollary \ref{cor:L2 estimate}). In the next section, we give the proofs of Corollary \ref{cor:green form} and \ref{cor:good sobolev}.
 The proof of Theorem \ref{thm:complex L^p estimate} is provided in Section \ref{sec:estimate}.
 In the last section, we will prove some Sobolev-type inequalities for K\"ahler families (Theorem \ref{thm:sobolev family good})
 and give some lower bounds of the eigenvalues (Theorem \ref{thm:eigenvalue family}), which will complete the proof of Theorem \ref{thm:family version}. 
  \section{Notations and Conventions}\label{notations}

    \par In this section, we fix some notations and conventions that are needed in our discussions.\\
    \indent We set $\mathbb{N}:=\{0,1,2,\cdots\}.$
    Let $T$ be a symmetric linear operator on an inner product space $(V,\langle\cdot,\cdot\rangle)$, we write $T\geq 0$ if 
$\langle Tx,x\rangle\geq 0$ for all $x\in V$, and  say $T$ is nonnegative. For  $b\in\mr$, we write $T\geq b$ if 
$T-bI$ is nonnegative, where $I$ is the identity operator.\\ %Set $\mr_+:=\{x\in\mr|\ x>0\}$.
 %  For any subset $A$ in a topological space $X$,
 %   we use $\overline{A}$ to denote its closure in $X$, $A^\circ$ to denote the set of its interior point,
  %  and set $\partial A:=\overline A\setminus A^\circ.$ If $A,B$ are two subsets of a topological space, then we write $A\subset\subset B$ if $\overline{A}$ is a compact subset of $B.$  
    \indent  Let $M$ be a complex manifold of complex dimension $n$, the Riemannian distance on $M$ is denoted by $d(\cdot,\cdot)$, 
    and the Riemannian volume form is denoted by $dV$. For $x\in M$ and $r>0$, set 
    $$B(x,r):=\{y\in M|\ d(x,y)<r\}.$$
    For a subset $U\subset M$, we use $|U|$ (resp. $\diam(U)$) to denote the volume (resp. diameter) of $U$. When $\omega$ is a K\"ahler metric on $M$, then  
    $$dV=\frac{\omega^n}{n!}.$$
      For a Hermitian vector bundle $E$ over $M$, the Hermitian metric is also denoted by $\langle \cdot,\cdot\rangle$, 
      which we assume to be complex linear in its first entry. 
      For any  $k\in\mathbb{N}\cup\{\infty\}$, and any open subset $U\subset\subset M$, we use $C^k(U,E)$ to denote the space of all $C^k$-smooth sections of $E$ over $U$.  Similarly, for $1\leq p\leq \infty$, $L^p(U,E)$ %(resp. $L_{\loc}^2(U,E)$)
   denotes of all Lebesgue measurable sections of $E$ over $U$ which
 are $L^p$-integrable. %(resp. locally $L^p$-integrable) 
 If $E=M\times \mc$ is trivial, then we write $C^k(U):=C^k(U,E)$.

    \par On a complex manifold $M$, the Laplace-Beltrami operator on $M$ is denoted by $\Delta$. We use $\bar\partial$ to denote the dbar operator on $M$. 
    %The Hodge-Laplace operator   of $\partial$ and $\bar\partial$ are given by 
    % $$\Delta_\partial:=\partial\partial^*+\partial^*\partial,\ \Delta_{\bar\partial}:=\bar\partial\bar\partial^*+\bar\partial^*\bar\partial.$$
     For $p,q\in \mathbb{N}$, let $\Lambda^{p,q}T^*M$ denote the bundle of smooth $(p,q)$-forms on $M$, and let $\Lambda^{p,q}T^*M\otimes E$  denote the bundle of smooth $E$-valued $(p,q)$-forms on $M$, where $E$ is a Hermitian holomorphic vector bundle over $M$. The operator $\bar\partial$ extends naturally to smooth $E$-valued $(p,q)$-forms, and we use the same notation $\bar\partial$ for this extension. Let $\nabla$ denote the Chern connection on $E$.
      The formal adjoint of $\bar\partial,\ \nabla$ is  denoted by $\bar\partial^*$ and $\nabla^*$, respectively. The $\bar\partial$-Laplacian is defined as  
       $$\square_{p,q}:=\bar\partial\bar\partial^{*}+\bar\partial^{*}\bar\partial.$$
       Now assume $M$ is K\"ahler. Define the Weitzenb\"ock curvature operator 
       $$\mathfrak{Ric}_{p,q}^E:=2\square_{p,q}-\nabla^*\nabla.$$
       For the local expression of  $\mathfrak{Ric}_{p,q}^E$, see \cite[Theorem 3.1]{LX10}.
       From the local expression, it follows that if $E$ is a trivial line bundle, then $\mathfrak{Ric}_{0,0}^E=0$ and $\mathfrak{Ric}_{n,0}^E$ is the scalar curvature of  $M$. The Bochner-Weitzenb\"ock formula states that 
      % Choose a local orthornormal frame $V_1,\cdots,V_n,\bar{V_1},\cdots,\bar{V_n}$, 
      % the (second) Ricci curvature $\operatorname{Ric}^E$ of $E$ is defined by 
      % $$\operatorname{Ric}^Ef:=\sum_{i=1}^n R(V_i,\bar{V_i})f,\ \forall f\in C^2(\Lambda^{p,q}T^*M\otimes E).$$
       %Similarly, we can define $\mathfrak{Ric}_{p,q}^E$ as in (\ref{equ:operator}). 
       %Now suppose 
       %$M$ is K\"ahler. For the explicit expression of the curvature operator $\mathfrak{Ric}_{p,q}$, see \cite[Theorem 3.1]{LX10}.
       %Then we have 
       %$$2\square_{p,q}=D^*D+\operatorname{Ric}^E$$  
       %and 
       $$\Delta\frac{|f|^2}{2}=\operatorname{Re}\langle (-2\square_{p,q}+\mathfrak{Ric}_{p,q}^E)f,f\rangle+|\nabla f|^2,\ \forall f\in C^2(M,\Lambda^{p,q}T^*M\otimes E),$$
       where $\operatorname{Re}(\cdot)$ denotes the real part of a complex number.
\section{Estimates of eigenvalues, eigensections and heat kernels}\label{sec:main theorem}
This section provides estimates for eigenvalues, eigensections, and heat kernels, including the proof of Theorem \ref{thm: heat kernel}, and a lower bound for the first nonzero eigenvalue of $\square_{p,q}$, i.e., the proof of Theorem \ref{thm: first eigenvalue}. \\
\indent We begin with a $C^0$-estimate for eigensections.
\begin{lem}\label{lem:eigensection}
Under the same assumptions and notations of Theorem \ref{thm: heat kernel}. Let \(\lambda \geq 0\) be an eigenvalue of $\square_{p,q}$, and 
let \(\phi\) be a corresponding  eigensection, i.e., 
\begin{equation*}
\square_{p,q}\phi = \lambda\phi.
\end{equation*}
 Then we have the following mean value inequality:
\begin{equation*}
\sup_{M} |\phi|^{2} \leq (2\beta(2\lambda+K)+2\gamma)^{\frac{\alpha}{\alpha-1}}\cdot \alpha^{\frac{\alpha}{(\alpha-1)^2}}\int_{M} |\phi|^{2}.
\end{equation*}
\end{lem}
\begin{proof}
This proof is inspired by the proof of \cite[Lemma 2.2]{LZZ21}.
By Bochner-Weitzenb\"ock formula, we get 
$$\Delta\frac{|\phi|^2}{2}=\operatorname{Re}\langle (-2\square_{p,q}+\mathfrak{Ric}_{p,q}^E)\phi,\phi\rangle+|\nabla\phi|^2
\geq -(2\lambda+K)|\phi|^2.$$
Set $v:=|\phi|^2$, then 
$$\Delta v\geq -2(2\lambda+K)v.$$
For $p\geq 1$, integration by parts yields 
$$\int_{M}\frac{2p-1}{p^2}|\nabla v^p|^2=-\int_{M}v^{2p-1}\Delta v
\leq 2\int_{M}(2\lambda+K)v^{2p}.$$
The Sobolev-type inequality (\clubtag) then implies
$$\left( \int_{M} v^{2p\alpha} \right)^{\frac{1}{\alpha}} \leq pC\int_{M} v^{2p},$$
where $C:=2\beta(2\lambda+K)+2\gamma$.
Now setting $p:=\alpha^{k-1}$ for $k=1,2,3\cdots,$ we obtain by iteration: 
\begin{align*}
\left( \int_{M} v^{2\alpha^{k}} \right)^{1/\alpha^{k}} &\leq C^{\alpha^{-(k-1)}} \alpha^{(k-1)\alpha^{-(k-1)}} \left( \int_{M} v^{2\alpha^{k-1}} \right)^{1/\alpha^{k-1}}\\
&\leq C^{\sum_{i=1}^k \alpha^{-(i-1)}}\prod_{j=1}^k \alpha^{(j-1)\alpha^{-(j-1)}}\int_{M} v.
\end{align*}
Taking the limit as $k\rw \infty$ gives the result.
\begin{comment}
Let $k\rw \infty$, then we conclude the result.
Since $\Delta=-2\square_{0,0}$, then we get 
$$\int_{M}\frac{1}{2}|\phi|^{2p-2}\Delta|\phi|^2\geq 
\int_{M}\left[ |\phi|^{2p-2} |\nabla|\phi||^{2} -(\lambda+K)|\phi|^{2p} \right].$$

Using integration by parts on the LHS, one gets that for \(p\geq 1\),
\[
-\int_{M} \frac{2(p-1)}{p^{2}} |\nabla|\phi|^{p}|^{2} \geq \int_{M} \left[ \frac{1}{p^{2}} |\nabla|\phi|^{p}|^{2} -(\lambda+K)|\phi|^{2p} \right],
\]
i.e., 
\[
\int_{M} |\nabla|\phi|^{p}|^{2} \leq \frac{p^{2}(\lambda + K)}{2p - 1} \int_{M} |\phi|^{2p}.
\] 
By assumption, %and by approximation (see Theorem 3.1 of nonlinear analysis), 
we know 
$$\left( \int_{M} |\phi|^{2p\alpha} \right)^{1/\alpha}\leq pC\int_{M} |\phi|^{2p},$$
\end{comment}
\end{proof}
Lemma \ref{lem:eigensection} easily implies the $L^k$-boundedness of the Bergman projection.
\begin{cor}\label{cor:Bergman}
 Under the same assumptions and notations of Theorem \ref{thm: heat kernel}, then the Bergman projection $\mathcal{P}$ is 
 $L^k$-bounded for any $1\leq k\leq \infty$. 
\end{cor}
\begin{remark}
If $\mathfrak{Ric}_{p,q}^E\geq 0,\mathfrak{Ric}_{p,q\pm 1}^E\geq 0$ and $1<k<\infty$, then 
Corollary \ref{cor:Bergman} also follows from \cite[Theorem 6.2]{LX10}.
\end{remark}
\indent Using \cite[Lemma 4.1]{B02}, the following dimensional estimate of Dolbeault cohomology groups can be derived from Lemma \ref{lem:eigensection}.
\begin{cor}\label{cor:estimate1}
Under the same assumptions and notations of Theorem \ref{thm: heat kernel}, and set 
$$\mathcal{H}_{\leq\lambda}^{p,q}(M,E):=\{\phi\in L^2(M,\Lambda^{p,q}T^*M\otimes E)|\ \square_{p,q}\phi=\mu\phi\text{ for some }0\leq\mu\leq \lambda\}.$$
Then 
$$\dim_{\mc}\mathcal{H}_{\leq\lambda}^{p,q}(M,E)\leq r|M|(2\beta(2\lambda+K)+2\gamma)^{\frac{\alpha}{\alpha-1}}\cdot \alpha^{\frac{\alpha}{(\alpha-1)^2}}.$$
\end{cor}

\begin{proof}
By regularity, every element of $\mathcal{H}_{\leq\lambda}^{p,q}(M,E)$ is smooth, then the result follows from 
Lemma \ref{lem:eigensection} and \cite[Lemma 4.1]{B02}.
\end{proof}
From Corollary \ref{cor:estimate1}, we have the following:
\begin{cor}\label{cor:estimate2}
Under the same assumptions and notations of Theorem \ref{thm: heat kernel}, and we further assume $E$ is a line bundle.
Suppose $F$ is a Hermitian holomorphic vector bundle of rank $s$ over $M$, then there are constants $K_1,K_2\geq 0$ such that 
$$\mathfrak{Ric}_{p,q}^{E^{\otimes \ell}\otimes F}\geq -\ell K_1-K_2,\ \forall \ell\in\mathbb{N}\setminus\{0\}.$$
Then for any $\ell\in\mathbb{N}\setminus\{0\}$ and any $\lambda\geq 0$, we have 
$$\dim_{\mc}\mathcal{H}_{\leq\lambda}^{p,q}(M,E^{\otimes \ell}\otimes F)\leq s|M|(2\beta(2\lambda+\ell K_1+K_2)+2\gamma)^{\frac{\alpha}{\alpha-1}}\cdot \alpha^{\frac{\alpha}{(\alpha-1)^2}}.$$
\end{cor}
\indent Note that the assumption $M$ is K\"ahler in Corollary \ref{cor:estimate2} is not essential, and can be removed by a suitable modification.
Since $\alpha$ can be taken to $n/(n-1)$ when $n\geq 2$, then for $(p,q)=(n,0)$, Corollary \ref{cor:estimate2} can be regarded as a generalization and a  more precise version of \cite[Theorem 1.1]{B02}. \\
\indent Similar to Lemma \ref{lem:eigensection}, we can derive a $C^1$-estimate for eigensections.
\begin{lem}\label{lem:eigensection derivative}
Under the same assumptions and notations of Theorem \ref{thm: heat kernel}. Let \(\lambda \geq 0\) be an eigenvalue of $\square_{p,q}$, and \(\phi\) be a
corresponding eigensection.
 Then 
\begin{equation*}
\sup_{M} (|\bar\partial\phi|^{2}+|\bar\partial^*\phi|^{2}) \leq \lambda (2\beta(2\lambda+K_0)+2\gamma)^{\frac{\alpha}{\alpha-1}}\cdot \alpha^{\frac{\alpha}{(\alpha-1)^2}}\int_{M} |\phi|^{2},
\end{equation*}
where $K_0:=\max\{K_+,K_{-}\}.$
\end{lem}
\begin{proof}
Note that the commutation relations
$$\square_{p,q}\bar\partial=\bar\partial\square_{p,q},\ \square_{p,q}\bar\partial^*=\bar\partial^*\square_{p,q}.$$
Set 
$$v:=|\bar\partial\phi|^2+|\bar\partial^*\phi|^2.$$
By the Bochner-Weitzenb\"ock formula,   
$$\Delta v\geq -2(2\lambda+K_0)v.$$
Following the proof of Lemma \ref{lem:eigensection}, we obtain 
$$\sup|v|\leq (2\beta(2\lambda+K_0)+2\gamma)^{\frac{\alpha}{\alpha-1}}\cdot \alpha^{\frac{\alpha}{(\alpha-1)^2}}\int_{M} v.$$
Integration by parts gives 
$$\int_M v=\lambda\int_{M}|\phi|^2,$$
completing the proof.
\end{proof}
\begin{remark}
 When $E$ is trivial and $(p,q)=(n,0)$, we cannot estimate $\nabla\phi$ directly
 as in \cite[Lemma 2.3]{LZZ21} without  a bound on  the Ricci curvature of $M$.  However, the estimates 
 in Lemma \ref{lem:eigensection derivative} suffice for our purposes.
\end{remark}

\indent To prove Theorem \ref{thm: heat kernel}, we need the following $C^1$-estimates for  solutions to the heat equation
associated with $\square_{p,q}$.
\begin{lem}\label{lem:  the solution of heat equation}
    Under the same assumptions and notations of Theorem \ref{thm: heat kernel}. Let  $\phi(t,x)$ be a $C^3$-smooth $E$-valued $(p,q)$-form
    satisfying the heat equation on $(0,T]\times M$ for some $T>0$,  i.e.,
   \begin{equation*}
      \left(\frac{\partial}{\partial t}+\square_{p,q}\right) \phi =0.
   \end{equation*}
   For any fixed point $O\in M$, the following estimate holds:
     
   \begin{align*}
      &\quad \sup_{[63T/64, T]\times B_{\sqrt{T}/8}}\left(|\bar\partial_x\phi(t,x)|^2+|\bar\partial_x^*\phi(t,x)|^2\right)\\
      & \leq C|B_{\sqrt{T}/2}|\left(KT+1\right)
\left(K_0+1+\frac{\gamma}{\beta}+\frac{1}{T}\right)^{\frac{2\alpha-1}{\alpha-1}}\sup_{[3T/4, T]\times B_{\sqrt{T}/2}} |\phi|^2,
   \end{align*}
   where 
   $$C:=16^{\frac{20\alpha^2}{(\alpha-1)^2}}\beta^{\frac{\alpha}{\alpha-1}}n,\ K_0:=\max\{K_+,K_{-}\},\ B_s:=B(O,s).$$
\end{lem}

\begin{proof}
 This proof is motivated by the proof of \cite[Proposition 3.3]{LZZ21}.
   In the following, only the operator 
  $\partial_t$ acts on the time variable $t$, and all other operators act on the spatial variable. 
 Fix a $T>0$.
 For positive numbers $r,\delta,\sigma$ such that 
 $$2r<\sqrt{T},\ \delta<\sigma\leq 4,$$
  let $\eta\in C^\infty(\mr)$ be a cut-off function satisfying 
  $$0\leq \eta\leq 1,\ \eta \equiv  0\text{ on }[0, T - \sigma r^2],\ \eta \equiv 1\text{ on }[T - \delta r^2, T],\ |\eta'| \leq \frac{2}{(\sigma - \delta)r^2}.$$
   For $0<\mu<\nu\leq 2$, let $\xi\in C^\infty(\mr)$ be a cut-off function 
  $$0\leq \xi\leq 1,\ \xi \equiv  1\text{ on }[0, \mu r],\ \xi=0\text{ on }[\nu r,\infty),\ |\xi'| \leq \frac{2}{(\nu - \mu)r}.$$
  Set $\psi(t,x) := \xi(d(O,x))\eta(t).$
  
  Set 
$$v:=|\bar\partial\phi|^2+|\bar\partial^*\phi|^2,$$
 then by the Bochner-Weitzenb\"ock formula and the assumption, we get  
$$\left(\frac{\partial}{\partial t}-\frac{\Delta}{2}\right)v\leq K_0v.$$
For $p\geq 1$, this implies 
\begin{equation}\label{equ: v^p-1}
\left(\frac{\partial}{\partial t}-\frac{\Delta}{2}\right)v^p\leq pK_0v^p.
\end{equation}
 Multiplying both sides of \eqref{equ: v^p-1} by $\psi^2 v^p$ and taking integral over $(0,T]\times M$ yields
  \begin{equation}\label{equ: v^p-1  2}
    \int_{(0,T]\times M} \psi^2 v^p \left(\frac{\partial}{\partial t}-\frac{\Delta}{2}\right) v^p \leq pK_0\iint_{(0,T]\times M} \psi^2 v^{2p},
  \end{equation}
  where we have omitted the volume element for simplicity of notations. 
On the other hand, integration by parts gives 
\begin{align*}
&\quad \iint_{(0,T]\times M}\psi^2v^p\left(\frac{\partial}{\partial t}-\frac{\Delta}{2}\right)v^p+\iint_{(0,T]\times M}\psi v^{2p}\frac{\partial\psi}{\partial t}\\
&=\frac{1}{2}\int_{M}\psi(T,x)^2|v(T,x)|^{2p}+\frac{1}{2}\iint_{(0,T]\times M}\left(|\nabla(\psi v^p)|^2-v^{2p}|\nabla\psi|^2\right),
\end{align*}
then we get 
\begin{align}\label{equ: v^p-1  3}
&\quad \frac{1}{2}\int_{M}\psi(T,x)^2|v(T,x)|^{2p}+\frac{1}{2}
\iint_{(0,T]\times M}|\nabla(\psi v^p)|^2\\
&\leq pK_0\iint_{(0,T]\times M}\psi^2v^{2p}+\iint_{(0,T]\times M}v^{2p}\left(\psi\frac{\partial\psi}{\partial t}+|\nabla\psi|^2\right).\nonumber
\end{align}
Since $T$ was arbitrary, then by the definition of $\psi$, we know 
\begin{align}\label{equ:definition}
&\quad \max\left\{\frac{1}{2}\sup_{t\in [T-\delta r^2,T]}\int_{B_{\mu r}}|v(t,x)|^{2p},\frac{1}{2}\iint_{(0,T]\times M}|\nabla(\psi v^p)|^2\right\}\\
&\leq \left(pK_0+\frac{4}{(\sigma - \delta)r^2} + \frac{4}{(\nu - \mu)^2 r^2} \right)\iint_{[T-\sigma r^2,T]\times B_{\mu r}}v^{2p}\nonumber.
\end{align}
Applying the Sobolev-type inequality \eqref{ineq:club}, we obtain 
$$\left(\int_{M} (\psi v^p)^{2\alpha }\right) ^{1/\alpha} \leq \beta\int_{M} |\nabla (\psi v^p)|^{2} + \gamma \int_{M} (\psi v^p)^{2 }, $$
i.e. 
$$\left(\int_{B_{\mu r}} (\psi v^p)^{2\alpha }\right) ^{1/\alpha} \leq \beta\int_{B_{\nu r}} |\nabla (\psi v^p)|^{2} + \gamma \int_{B_{\nu r}} v^{2p}. $$
Combined with \eqref{equ:definition}, this yields 
\begin{align*}
&\quad \int_{0}^T\left(\int_{B_{\mu r}} (\psi v^p)^{2\alpha }\right) ^{1/\alpha}ds\\
&\leq \left(2p\beta K_0+\gamma+ \frac{8\beta}{(\sigma - \delta)r^2} + \frac{8\beta}{(\nu - \mu)^2 r^2} \right)\iint_{[T-\sigma r^2,T]\times B_{\nu r}}v^{2p}. 
\end{align*}
Using H\"older's inequality and \eqref{equ:definition}, we know 
\begin{align}
&\quad \iint_{[T-\delta r^2, T]\times B_{\mu r}}(\psi^2v^{2p})^{\left(2-\frac{1}{\alpha}\right)}\\
&\leq \int_{T-\delta r^2}^T \left(\int_{B_{\mu r}}(\psi v^{p})^{2\alpha}\right)^{1/\alpha}
\left(\int_{B_{\mu r}}\psi^2v^{2p}\right)^{1-1/\alpha}ds\nonumber\\
&\leq \int_{T-\delta r^2}^T \left(\int_{B_{\mu r}}(\psi v^{p})^{2\alpha}\right)^{1/\alpha}ds\cdot \sup_{[T-\delta r^2,T]}\left(\int_{B_{\mu r}}v^{2p}\right)^{1-1/\alpha}
\nonumber\\
&\leq  C_0p^{2}\left(\int_{[T-\delta r^2,T]\times B_{\nu r}}v^{2p}\right)^{2-1/\alpha},\nonumber
\end{align}
where 
$$C_0:=32\beta \left(K_0+1+\frac{\gamma}{\beta}+\frac{1}{(\sigma - \delta)r^2} + \frac{1}{(\nu - \mu)^2 r^2}\right)^{2-1/\alpha}.$$
By the definition of $\psi$ again, we know 
\begin{align}\label{equ:comparison}
\iint_{[T-\delta r^2, T]\times B_{\mu r}}(v^{2p})^{\left(2-\frac{1}{\alpha}\right)}
\leq C_0p^{2}\left(\int_{[T-\sigma r^2,T]\times B_{\nu r}}v^{2p}\right)^{2-1/\alpha}.
\end{align}
Set
$$\theta:=2-\frac{1}{\alpha}.$$
For  $ 0<\tau\leq 1,$ and $k=1,2\cdots,$ set 
$$p:=\theta^{k-1},\ \sigma:=1+ \frac{(1+\tau)^2 - 1}{2^{k-1}},\ \delta := 1 + \frac{(1+\tau)^2 - 1}{2^k},$$
$$\mu:=1 + \frac{\tau}{2^k},\ \nu:=1 + \frac{\tau}{2^{k-1}},\ r_k:=(1 + \frac{\tau}{2^k})r\geq r,\ I_k:=[T-\delta r^2,T].$$
From \eqref{equ:comparison}, we conclude 
\begin{align}
\iint_{I_k\times B_{r_k}}v^{2\theta^k}
\leq C_1\theta^{2k}16^k\left(\int_{I_{k-1}\times B_{r_{k-1}}}v^{2\theta^{k-1}}\right)^{\theta},
\end{align}
where
$$C_1:=32\beta \left(K_0+1+\frac{\gamma}{\beta}+\frac{2}{\tau^2 r^2}\right)^{2-1/\alpha}.$$
Applying Moiser iteration as in Lemma \ref{lem:eigensection}, we obtain 
\begin{equation}\label{equ:comparision1}
\sup_{[T - r^2, T]\times B_r} v^2 \leq C_2
\iint_{[T-(r+\tau r)^2,T]\times B_{r+\tau r}}v^2,
\end{equation}
where 
$$C_2:=16^{\frac{5\alpha^2}{(\alpha-1)^2}}
\beta^{\frac{\alpha}{\alpha-1}}\left(K_0+1+\frac{\gamma}{\beta}+\frac{2}{\tau^2 r^2}\right)^{\frac{2\alpha-1}{\alpha-1}}.$$
Next, for $j=0,1,2\cdots,$ set
$$r_j := (\sum_{k=0}^j 2^{-k})r\geq r,\ Q_j := [T - r_j^2, T]\times B_{r_j},\ \tau := 2^{-(j+1)},$$
then 

$$Q_0 = [T - r^2, T]\times B_r \subset Q_1 \subset \cdots  \subset [T - (2r)^2, T]\times B_{2r}.$$ 
By \eqref{equ:comparision1}, 
\begin{align*}
   \sup_{Q_j} v^2 \leq C_34^{\frac{2(j+2)\alpha}{(\alpha-1)}}\iint_{Q_{j+1}} v^2\leq C_34^{\frac{2(j+2)\alpha}{(\alpha-1)}}\left(\sup_{Q_{j+1}} v^2\right)^{\frac{1}{2}} \iint_{[T - (2r)^2, T]\times B_{2r}} v,
\end{align*}
where 
$$C_3:=16^{\frac{5\alpha^2}{(\alpha-1)^2}}\beta^{\frac{\alpha}{\alpha-1}}\left(K_0+1+\frac{\gamma}{\beta}+\frac{1}{r^2}\right)^{\frac{2\alpha-1}{\alpha-1}}.$$
Iterating again,  and letting $j\rightarrow \infty$, we get
\begin{equation*}
\sup_{[T-r^2,T]\times B_r} v^2 \leq 16^{\frac{6\alpha}{\alpha-1}}C_3^2\left(\iint_{[T - (2r)^2, T]\times B_{2r}} v\right) ^2,
\end{equation*}
i.e.,
\begin{equation}\label{equ:1key}
\sup_{[T-r^2,T]\times B_r} v \leq 16^{\frac{3\alpha}{\alpha-1}}C_3\iint_{[T-(2r)^2,T]\times B_{2r}} v.
\end{equation}
\indent At last, let us bound the right hand side of the previous inequality.
 Set $r:=\sqrt{T}/4$, let $\chi\in C^\infty(\mr)$ be a cut-off function such that
 $$0\leq \chi\leq 1,\ \chi \equiv  1\text{ on }[0, r],\ \chi=0\text{ on }[2r,\infty),\ |\chi'| \leq \frac{2}{r}.$$
Set $\rho(x):=\chi(d(O,x)).$
Using integration by parts, Kato's inequality, and Cauchy-Schwarz inequality, we get 
\begin{align}\label{equ:parts}
\int_{M}\Delta|\phi|^2\rho^2&=-4\int_{M}\rho|\phi|\nabla|\phi|\cdot\nabla\rho\leq 4\int_{M}\rho|\phi|\cdot |\nabla\phi|\cdot|\nabla\rho|\\
&\leq \int_{M}|\nabla\phi|^2\rho^2+4\int_{M}|\phi|^2|\nabla\rho|^2.\nonumber
\end{align}
The Bochner-Weitzenb\"ock formula and the assumption $\mathfrak{Ric}_{p,q}^E\geq -K$ imply 
$$\left(\frac{\partial}{\partial t}-\frac{\Delta}{2}\right)|\phi|^2\leq K|\phi|^2-|\nabla\phi|^2.$$ 
Using \eqref{equ:parts} and integration by parts, we get 
\begin{align*}
   0&\geq \iint_{[3T/4, T]\times M}\left(\rho^2\left(\frac{\partial}{\partial t} - \frac{\Delta}{2}\right) |\phi|^2-K\rho^2|\phi|^2+\rho^2|\nabla\phi|^2\right)\\
   &\geq \iint_{[3T/4, T]\times M}\rho^2\frac{\partial}{\partial t}|\phi|^2\\
   &\quad  +\iint_{[3T/4, T]\times M}\left(\frac{1}{2}|\nabla\phi|^2\rho^2-2|\phi|^2|\nabla\rho|^2-K\rho^2|\phi|^2\right)\\
  &\geq \iint_{[3T/4, T]\times M}\left(\frac{1}{2}|\nabla\phi|^2\rho^2-2|\phi|^2|\nabla\rho|^2-K\rho^2|\phi|^2\right)-\int_{M}|\phi(3T/4,x)|^2\rho^2.
\end{align*}
According to Lemma 6.8 of \cite{GM75} (see also \cite[Lemma 4.1]{EGHP23}), we have  
  $$|\bar\partial\phi|^2+|\bar\partial^*\phi|^2\leq 2n|\nabla\phi|^2.$$
so we have (by definition of $\rho$)
\begin{align}\label{equ:2key}
   &\quad \frac{1}{4n}\iint_{[3T/4, T]\times B_{\sqrt{T}/4} }\left(|\bar\partial\phi|^2+|\bar\partial^*\phi|^2\right)\\
   &\leq \iint_{[3T/4, T]\times B_{\sqrt{T}/2}}\left(\frac{128|\phi|^2}{T}+K|\phi|^2\right)+\int_{B_{\sqrt{T}/2}}|\phi(3T/4,x)|^2.\nonumber
\end{align}

Combining \eqref{equ:1key} and \eqref{equ:2key} yields

  $$\sup_{ [63T/64, T]\times B_{\sqrt{T}/8}}\left(|\bar\partial\phi|^2+|\bar\partial^*\phi|^2\right) \leq C_4\sup_{[3T/4, T]\times B_{\sqrt{T}/2}} |\phi|^2,$$
where 
$$C_4:=16^{\frac{20\alpha^2}{(\alpha-1)^2}}\beta^{\frac{\alpha}{\alpha-1}}n|B_{\sqrt{T}/2}|\left(KT+1\right)
\left(K_0+1+\frac{\gamma}{\beta}+\frac{1}{T}\right)^{\frac{2\alpha-1}{\alpha-1}}.$$
\end{proof}
\begin{remark}
The compactness assumption on $M$ is not necessary for  Lemma \ref{lem:  the solution of heat equation}. 
\end{remark}
Lemma \ref{lem:  the solution of heat equation} implies  the following $C^1$-estimates for harmonic forms, which we think has some independent interest.
\begin{cor}\label{cor:harmonic form}
  Let $M$ be a complete K\"ahler manifold of complex dimension $n$ satisfying the Sobolev-type inequality (\clubtag).
  Fix $p,q\in\{0,\cdots,n\}.$ Let $E$ be a Hermitian holomorphic vector bundle over $M$ such that there are constants $K\geq 0, K_{+}\geq 0,
  K_{-}\geq 0$ satisfying 
  $$\mathfrak{Ric}_{p,q}^E\geq -K,\ \mathfrak{Ric}_{p,q+1}^E\geq -K_+,\ \mathfrak{Ric}_{p,q-1}^E\geq -K_{-}.$$
   Suppose $\phi$ is a harmonic $E$-valued $(p,q)$-form on $M$. Then for any ball $B(O,r)\subset M$, we have 
   $$\sup_{B_{r/4}}(|\bar\partial\phi|^2+|\bar\partial^*\phi|^2)
   \leq C|B_r|\left(4Kr^2+1\right)
\left(K_0+1+\frac{\gamma}{\beta}+\frac{1}{r^2}\right)^{\frac{2\alpha-1}{\alpha-1}}\sup_{B_r}|\phi|^2,$$
   where 
   $$K_0:=\max\{K_+,K_{-}\},\ C:=100^{\frac{40\alpha^2}{(\alpha-1)^2}}\beta^{\frac{\alpha}{\alpha-1}}n,\ B_s:=B(O,s).$$
\end{cor}
\begin{remark}
If $\phi$ is harmonic only harmonic in $B(O,r)$, a similar conclusion of Corollary \ref{cor:harmonic form} holds,
please see \cite[Theorem 3.1]{DHQ25}. 
\end{remark}

Now we prove Theorem \ref{thm: heat kernel}.
\begin{proof}
   (i) By \cite[Theorem 4.3]{LX10}, we have 
   \begin{equation}\label{estimate1}
   |H_{p,q}(t,x,y)|\leq e^{Kt}H(t,x,y),\ \forall (t,x,y)\in \mr_+\times M\times M,
   \end{equation}
   where $H(t,x,y)$ is the heat kernel for functions.
   Now it is well known Sobolev-type inequality (\clubtag)  implies a Gaussian upper bound for $H$. This idea is due to \cite{D90}
 (see also \cite{SC02} and \cite{Z11}). For completeness, we sketch a proof here. The argument follows 
   the presentation of \cite[Section 5]{GPSS23}.\\
  \indent  By approximation, we may assume $\gamma\neq 0$. Let $u_0 \in C^\infty(M)$ be an arbitrary positive smooth function on $M$. For a given $b \in \mathbb{R}$ and an arbitrary Lipschitz function $\phi$ with $\|\phi\|_{L^\infty(M)} \leq 1$, we consider the function $u(t,x)$ defined by
    \begin{equation}\label{equ: u_0}
    u(t,x) := e^{-b \phi(x)} \int_M H(t,x,y) e^{b \phi(y)} u_0(y) dV(y),
    \end{equation}
    which is positive and satisfies the equation
    
    \begin{equation}
    \frac{\partial u}{\partial t} = e^{-b \phi} \Delta(u e^{b \phi}), \quad u(0,x) = u_0(x),\ \forall x\in M.
    \end{equation}
    Fix any time $T > 0$. We define an increasing, piecewise smooth, and continuous function $r(t)$ of $t \in (0,T)$ by
    
    $$
    r = r(t) =
    \begin{cases}
    \frac{2^{1/2} - 1}{(T/2)^{1/2}} t^{1/2} + 1, & \text{if } t \in [0, T/2), \\
    \frac{T^{1/2}}{(T-t)^{1/2}}, & \text{if } t \in [T/2, T),
    \end{cases}
    $$
    which satisfies $r(0) = 1$ and $r(T^-) = \infty$. Denote $\|u\|_r$ to be the $L^r(M)$-norm of $u$. By straightforward calculations, we have
    (see \cite[Formula (5.6)]{GPSS23})
   \begin{equation}\label{eqn: the estimate of log u}
     \frac{\partial}{\partial t} \ln \|u\|_r \leq \frac{4(1-r)}{r^2} \int_M |\nabla v|^2 + \frac{r'}{r^2} \int_M v^2 \ln v^2
     +\frac{b^2(r-2)^2}{2(r-1)}+b^2.
   \end{equation}
where 
$$v := \frac{u^{r/2}}{\|u^{r/2}\|_2},\ r':=\frac{dr}{dt},$$
 and $\delta \in (0,1)$ is a constant to be chosen later. Note that $\|u^{r/2}\|_2^2 = \|u\|_r^r$ and $\|v\|_2 = 1$.
%The following inequality is elementary:
%$$\ln x\geq \sigma x-1-\ln \sigma,\ \forall \sigma>0,x>0.$$
We next apply Jensen's inequality and the Sobolev-type inequality (\clubtag): 
\begin{align*}
   \int_M v^2 \ln v^2 &= \frac{1}{\alpha-1} \int_M v^2 \ln\left( v^{2(\alpha -1)}\right)\\
   &=\frac{\alpha}{\alpha-1}\ln \left(\int_M v^{2\alpha }\right)^{\frac{1}{\alpha}}\\ 
   &\leq \frac{\alpha}{\alpha-1} \ln \left( \int_M \left(\beta  |\nabla v|^2 +\gamma v^2 \right) \right)\\
   &=\frac{\alpha}{\alpha-1} \ln \left(\beta\int_M|\nabla v|^2 +\gamma \right).
\end{align*}
Combining this with \eqref{eqn: the estimate of log u}, we get 

\begin{equation}\label{equ: the estimate of log u 2}
   \begin{aligned}
      \frac{\partial}{\partial t} \ln \|u\|_r &\leq \frac{4(1-r) }{r^2} \int_M |\nabla v|^2   + \frac{r'}{r^2} \frac{\alpha}{\alpha-1} \ln \left(\beta\int_M|\nabla v|^2 +\gamma \right)\\
      &\quad + \frac{b^2(r-2)^2}{2(r-1)}+b^2.
   \end{aligned} 
\end{equation}
Considering the following concave function 
$$
f(x) := \frac{4(1-r) }{r^2\beta } x + \frac{r'}{r^2} \frac{\alpha}{\alpha-1} \ln (x+\gamma),\ x\geq 0,
$$
achieves the maximum at the point
$$x := \frac{r' \alpha\beta }{4(r-1) (\alpha-1)} - \gamma,$$
then we get 
%%$$\frac{\partial}{\partial t} \ln \|u\|_r\leq \frac{r'}{r^2} \frac{\alpha}{\alpha-1} \ln \left( \frac{r' \alpha \beta}{4(r-1) (\alpha-1)\gamma} \right) + \frac{4(r-1) \gamma }{r^2\beta } - \frac{r'}{r^2} \frac{\alpha}{\alpha-1}+\frac{\alpha r'}{(\alpha-1)r^2}\ln\gamma.$$
\begin{equation}\label{equ:key inequality}
\frac{\partial}{\partial t} \ln \|u\|_r\leq \frac{r'}{r^2} \frac{\alpha}{\alpha-1} \ln \left( \frac{r'}{r-1} \right)+ \frac{4(r-1) \gamma }{r^2\beta } + \frac{r'}{r^2} \Psi+\frac{b^2(r-2)^2}{2(r-1)}+b^2,
\end{equation}
where the constant $\Psi$ is given by
$$
\Psi := \frac{\alpha }{\alpha-1} \ln \left(\frac{\alpha \beta }{4(\alpha-1)}\right) - \frac{\alpha}{\alpha-1}.
$$
On the other hand, by a direct computation, we know (see \cite{GPSS23})
$$\int_0^{T}\frac{r'}{r^2}\ln\left(\frac{r^2}{r-1}\right)dt\leq -\ln T+\frac{5}{2},$$
$$\int_0^T\frac{r'}{r^2}dt=1,\ \int_0^T\frac{(r-2)^2}{r-1}dt\leq \frac{21}{10}T,\ \int_0^T\frac{r-1}{r^2}dt\leq \frac{1}{5}T.$$
Integrating (\ref{equ:key inequality}), we know 
\begin{equation}\label{equ:estimate1}
\ln\left(\frac{\|u_T\|_\infty}{\|u_0\|_1}\right)\leq  \frac{\alpha}{\alpha-1}\left(-\ln T+\frac{5}{2}\right)
+\frac{4 \gamma T}{5\beta}+\Psi+\frac{41b^2}{20}T.
\end{equation}
where $u_T(x) := u(x,T)$. Since \eqref{equ:estimate1} holds for any positive function $u_0$, by duality and the representation formula \eqref{equ: u_0},
then we get 
\begin{equation}
H(x,y,T) \leq e^{\frac{3\alpha}{\alpha-1}}e^{\Psi} T^{-\frac{\alpha}{\alpha-1}}e^{\frac{4 \gamma T}{5\beta}+\frac{41b^2}{20}T} e^{b(\phi(x)-\phi(y))},\ \forall x,y\in M.
\end{equation}

For two fixed points $x,y \in M$, if we take 
$$b := \frac{10(\phi(y)-\phi(x))}{41}, $$
 we have 
\begin{equation}\label{equ: H_{0,0}(x,y,T)}
   H(T,x,y) \leq e^{\frac{3\alpha}{\alpha-1}}e^{\Psi} T^{-\frac{\alpha}{\alpha-1}}e^{\frac{4 \gamma T}{5\beta}} e^{-\frac{(\phi(x)-\phi(y))^2}{9T}}.
\end{equation}

Since \eqref{equ: H_{0,0}(x,y,T)} holds for any smooth function $\phi$ with $\|\nabla \phi\|_\infty \leq 1$, we can take a sequence of such functions which converge uniformly to $z \mapsto d(x,z)$. Therefore, from \eqref{equ: H_{0,0}(x,y,T)}, we conclude that for any $T > 0$, it holds that

\begin{equation}
   H(T,x,y) \leq e^{\frac{3\alpha}{\alpha-1}}e^{\Psi} T^{-\frac{\alpha}{\alpha-1}}e^{\frac{4 \gamma T}{5\beta}} e^{-\frac{(\phi(x)-\phi(y))^2}{9T}}.
\end{equation}
Combining this with \eqref{estimate1} gives the result.\\
   (ii) Suppose $0<\mu_1\leq\mu_2<\cdots$ are all nonzero eigenvalues of 
   $\square_{p,q}$, and $\phi_{b_{p,q}+1},\phi_{b_{p,q}+2},\cdots$ are the corresponding eigensections such that 
   they consist of an orthonormal basis of  the orthogonal complement of $\mathcal{H}_{p,q}(M,E)$ in $L^2(M,\Lambda^{p,q}T^*M\otimes E)$.
   By (i), there is a constant $C_0:=C_0(\alpha)>0$  such that 
   $$|H_{p,q}(t,x,x)|\leq C_0\beta^{\frac{\alpha}{\alpha-1}}e^{Kt+\frac{\gamma}{\beta}t}t^{-\frac{\alpha}{\alpha-1}},\ 
   \forall (t,x)\in\mr_+\times M.$$
   Similar to the proof of \cite[Corollary 4.6]{DL82}, it is easy to see that 
   \begin{equation}\label{equ:eigenvalues estimates}
   \mu_k\geq \frac{2c}{\beta}\left(r|M|e^{\frac{K}{\mu_1}+\frac{\gamma}{\beta\mu_1}}\right)^{\frac{1-\alpha}{\alpha}}k^{\frac{\alpha-1}{\alpha}},\ \forall k\geq 1,
   \end{equation}
   where $c>0$ is a constant only depending on $\alpha$.
   By Lemma \ref{lem:eigensection}, and shrinking $c$ if necessary, we have  
   $$\sup_M|\phi_{b_{p,q}+k}|^2\leq\frac{\theta}{c}\mu_k^{\frac{\alpha}{\alpha-1}},\ \forall k\geq 1,$$
   where 
   $$\theta:=\left(\beta+\frac{\beta K}{\mu_1}+\frac{\gamma}{\mu_1}\right)^{\frac{\alpha}{\alpha-1}}.$$ 
   Setting 
   $$d:=\sup_{x>0}e^{-\frac{x}{2}}x^{\frac{\alpha}{\alpha-1}}\leq \left(\frac{2\alpha}{\alpha-1}\right)^{\frac{\alpha}{\alpha-1}},\ \delta:=\frac{1}{\beta}\left(r|M|e^{\frac{K}{\mu_1}+\frac{\gamma}{\beta\mu_1}}\right)^{\frac{1-\alpha}{\alpha}},$$
   then we know 
   \begin{align*}
   &\quad \left|H_{p,q}(t,x,y)-\sum_{k=1}^{b_{p,q}}\phi_k(x)\otimes \phi_k(y)\right|\leq \sum_{k=1}^\infty \frac{\theta}{c}e^{-\mu_kt}\mu_k^{\frac{\alpha}{\alpha-1}}\\
   &=t^{\frac{\alpha}{1-\alpha}}\sum_{k=1}^\infty \frac{\theta}{c}e^{-\frac{\mu_k t}{2}}\left(e^{-\frac{\mu_kt}{2}}(\mu_kt)^{\frac{\alpha}{\alpha-1}}\right)\leq    t^{\frac{\alpha}{1-\alpha}}\sum_{k=1}^\infty \frac{d\theta}{c}e^{-c\delta k^{\frac{\alpha}{\alpha-1}}t}\\
   &\leq t^{\frac{\alpha}{1-\alpha}}\frac{d\theta}{c}\int_0^\infty e^{-c\delta x^{\frac{\alpha-1}{\alpha}}t}dx\leq C\theta \delta^{\frac{\alpha}{1-\alpha}}t^{\frac{2\alpha}{1-\alpha}},
   \end{align*}
   where $C$ is a constant only depending on $\alpha$.
   Clearly, 
   $$\theta \delta^{\frac{\alpha}{1-\alpha}}\leq r\beta^{\frac{2\alpha}{\alpha-1}}|M|e^{\left(\frac{K}{\mu_1}+\frac{\gamma}{\beta\mu_1}\right)
   \frac{2\alpha}{\alpha-1}}.$$
  (iii) We only consider the estimates of $\bar\partial_yH_{p,q}(t,x,y)$. Fix $x,y\in M$.  By definition, we clearly have
$$\bar\partial_y H_{p,q}(t,x,y)=\sum_{k=1}^\infty e^{-\mu_kt}\phi_{b_{p,q}+k}(x)\otimes \bar\partial_y\phi_{b_{p,q}+k}(y).$$
Thus, by Lemma \ref{lem:eigensection} and \ref{lem:eigensection derivative}, we know there is a constant $c:=c(\alpha)>0$ such that 
$$\sup_M|\phi_{b_{p,q}+k}|^2\leq\frac{\theta}{c}\mu_k^{\frac{\alpha}{\alpha-1}},\ \sup_M|\bar\partial\phi_{b_{p,q}+k}|^2\leq\frac{\theta}{c}\mu_k^{\frac{2\alpha-1}{\alpha-1}},\ \forall k\geq 1,$$
where 
$$\theta:=\left(\beta+\frac{\beta K_0}{\mu_1}+\frac{\gamma}{\mu_1}\right)^{\frac{\alpha}{\alpha-1}},\ K_0:=\max\{K,K_+,K_{-}\}.$$
Thus we know 
$$|\bar\partial_y H_{p,q}(t,x,y)|\leq \sum_{k=1}^\infty \frac{\theta}{c}e^{-\mu_kt}\mu_k^{\frac{\alpha}{\alpha-1}+\frac{1}{2}}.$$
Similar as the proof of (ii), we conclude there is a constant $C:=C(\alpha)>0$ such that 
$$|\bar\partial_y H_{p,q}(t,x,y)|\leq Cr\beta^{\frac{2\alpha}{\alpha-1}}|M|e^{\left(\frac{K}{\mu_1}+\frac{\gamma}{\beta\mu_1}\right)
   \frac{2\alpha}{\alpha-1}}t^{\frac{2\alpha}{1-\alpha}-\frac{1}{2}}.$$  
(iv)  We only consider the estimates of $\bar\partial_yH_{p,q}(t,x,y)$. Fix $x,y\in M, 0<t\leq 1$. By Lemma \ref{lem:  the solution of heat equation} and (i), 
there is a constant $C:=C(\alpha,\beta,\gamma,K,K_{+},K_{-})>0$ such that  
$$|\bar\partial_y H_{p,q}(t,x,y)|^2\leq Cn|B(y,\sqrt{t})|t^{\frac{4\alpha-1}{1-\alpha}}\sup_{(s,z)\in [3t/4, t]\times B(y,\sqrt{t}/2)} e^{-\frac{2d(x,z)^2}{9s}}.$$
Observation that 
$$\sup_{(s,z)\in [3t/4, t]\times B(y,\sqrt{t}/2)} e^{-\frac{2d(x,z)^2}{9s}}
\leq 3e^{-\frac{d(x,y)^2}{18t}},$$
then the conclusion follows.
 \end{proof}
 \indent In the last, we give an estimate of the first eigenvalue of $\square_{p,q}$.
 \begin{prop}\label{prop:first eigenvalue}
 Let $M$ be a compact K\"ahler manifold of complex dimension $n$.
  Fix $p,q\in\{0,\cdots,n\}.$ Let $E$ be a Hermitian holomorphic vector bundle over $M$ such that there is a real valued  function 
  $b$ over $M$ satisfying 
  $$\mathfrak{Ric}_{p,q}^E\geq b,\ \mu_1^0+\inf_{M}b> 0.$$
 where $\mu_1^0$ is the first nonzero eigenvalue of $\Delta$.
 Let $\lambda_1$ be the first eigenvalue of $\square_{p,q}$, then we have 
 $$2\lambda_1\geq \min\left\{\frac{1}{2}\mu_{1}^{0} + \inf_{M}b,\frac{1}{|M|}\int_{M}b,\frac{\mu_1^0}{2|M|}\int_{M}\frac{b}{\mu_1^0+b}\right\},$$
 where $\mu_1^0$ is the first nonzero eigenvalue of $\Delta$.
 \end{prop}

\begin{proof}
 This proof is motivated by the proof of \cite[Lemma 2.8]{LZZ21}.
 Suppose \(\phi\) is an eigenform of \(\lambda_{1}\), and we may assume 
 $$\int_{M}|\phi|^{2}=1$$
 By the Bochner-Weitzenb\"ock formula and Kato's inequality, we get 
    \begin{equation}
        \Delta\frac{|\phi|^{2}}{2}\geq -(2\lambda_1-b)|\phi|^2+|\nabla|\phi||^2.
    \end{equation}
    Integrating the above identity, we find that
    \begin{equation}\label{equ:inequlity 1}
            2\lambda_{1} \geq \int_{M}|\nabla |\phi||^{2} + \int_{M}b|\phi|^{2}.
    \end{equation}
   % where \(\lambda_{1,R}\) is the first eigenvalue of the scalar Schrodinger operator \(-\Delta^{\mathbb{R}}+R\) with \(\Delta^{\mathbb{R}}\) being the real Laplacian.

            If \(|\phi|\) is a constant function, then \eqref{equ:inequlity 1} gives us \\
            \begin{equation}
                2\lambda_{1} \geq \frac{1}{|M|}\int_{M}b.
            \end{equation}
            Thus we may assume that \(|\phi|\) is not constant. For simplicity of notations, we write 
            $$f := |\phi|,\ a := \frac{1}{|M|}\int_{M}f.$$ 
            We consider two cases. \\
    \textbf{Case 1:} Suppose
    \begin{equation}
        \int_{M}(f-a)^{2} \geq 1/2.
    \end{equation}
    Then \eqref{equ:inequlity 1} implies
    \begin{equation}
        \begin{split}
            2\lambda_{1} &\geq \int_{M}|\nabla f|^{2} + \int_{M}bf^{2} \\
            &\geq \mu_{1}^{0}\int_{M}(f-a)^{2} + \inf_{M}b\\
            &\geq \frac{1}{2}\mu_{1}^{0} + \inf_{M}b.
        \end{split}
    \end{equation}
    \textbf{Case 2:} Suppose
    \begin{equation}
        \int_{M}(f-a)^{2} < 1/2.
    \end{equation}
    Then we can expand the square to reach:
    \[
        \int_{M}f^{2} - 2a\int_{M}f + a^{2}|M| < 1/2,
    \]
    which shows that
    \begin{equation}
        a^{2} > \frac{1}{2|M|}.
    \end{equation}
    By  \eqref{equ:inequlity 1} again and use Cauchy-Schwarz inequality, we know 
    \begin{equation}
        \begin{split}
            2\lambda_{1} &\geq \int_{M}|\nabla f|^{2} + \int_{M}b(f-a+a)^{2} \\
            &\geq \int_{M}\left(\mu_{1}^{0} + b\right)(f-a)^{2} + 2a\int_{M}b(f-a) + a^{2}\int_{M}b\\
            &\geq a^2\int_{M}\frac{\mu_1^0b}{\mu_1^0+b}\\
            &\geq \frac{\mu_1^0}{2|M|}\int_{M}\frac{b}{\mu_1^0+b}.
        \end{split}
    \end{equation}
\end{proof}
Proposition \ref{prop:first eigenvalue} implies Theorem \ref{thm: first eigenvalue} and yields a vanishing theorem:
 \begin{cor}\label{cor:vanish}
 Let $M$ be a compact K\"ahler manifold of complex dimension $n$.
  Fix $p,q\in\{0,\cdots,n\}.$ Let $E$ be a Hermitian holomorphic vector bundle over $M$  such that there is a real valued  function 
  $b$ over $M$ satisfying 
  $$\mathfrak{Ric}_{p,q}^E\geq b\geq 0,$$
  and $b(x)>0$ at some point $x\in M$.
   Let $\lambda_1$ be the first eigenvalue of $\square_{p,q},$ then we have 
 $$\mu_1=\lambda_1\geq \frac{\mu_1^0}{4|M|}\int_{M}\frac{b}{\mu_1^0+b},$$
 where $\mu_1^0$ is the first nonzero eigenvalue of $\Delta$.
 In particular, $\mathcal{H}_{p,q}(M,E)=0$.
 \end{cor}
 \indent For some other estimates of the first eigenvalue and vanishing theorems using the Weitzenb\"ock curvature operator,
 please see \cite{RY94} and \cite{Z94}.\\
 \indent The following $L^2$-estimate of $\bar\partial$ can be easily derived from Corollary \ref{cor:vanish}, see
 \cite[Proposition 5.1]{B02} for example.
 \begin{cor}\label{cor:L2 estimate}
 Under the same assumptions and notations of Corollary \ref{cor:vanish} and we further assume $q\geq 1$.
 Then for $\bar\partial$-closed $f\in L^2(M,\Lambda^{p,q}T^*M\otimes E)$, we may find 
 $u\in L^2(M,\Lambda^{p,q-1}T^*M\otimes E)$ such that 
 $$\int_{M}|u|^2\leq \frac{5}{\sigma}\int_{M}|f|^2,$$
 where 
 $$\sigma:=\frac{\mu_1^0}{|M|}\int_{M}\frac{b}{\mu_1^0+b}.$$
 \end{cor} 
\section{Estimates of Green forms and Sobolev-type inequalities}
\indent In this section, we will give the proof of Corollary \ref{cor:green form} and Corollary \ref{cor:good sobolev}.\\
\indent Firstly, let us prove Corollary \ref{cor:green form}, and we restate it here for convenience.
\begin{cor}[= Corollary \ref{cor:green form}]
 Under the same assumptions and notations of Theorem \ref{thm: heat kernel}. Let $G_{p,q}$ be the Schwarz kernel of the Green operator for $\square_{p,q}$,
   %acts on $C^2(\overline\Omega,\Lambda^pT^*M\otimes E)$,
   then we have the following uniform estimates:   
   \begin{itemize}
     \item[(i)] There is a constant $C:=C(\alpha,\beta,\gamma,\mu_1,K,|M|)>0$ such that 
     $$|G_{p,q}(x,y)|\leq Cd(x,y)^{\frac{2}{1-\alpha}}+C,\ \forall (x,y)\in M\times M.$$
     \item[(ii)] Suppose the Ricci curvature of $M$ is bounded below by $-(2n-1)\theta$ for some constant $\theta\geq 0$,
     and  assume $\alpha$ satisfies 
     $$\frac{4\alpha-1}{\alpha-1}-n-2> 0,$$
      then there is a constant $C:=C(\alpha,\beta,\gamma,r,n,\mu_1,K,K_+,K_{-},M)>0$ such that 
      $$ \max\left\{|\bar\partial_yG_{p,q}(x,y)|,|\bar\partial^{*}_yG_{p,q}(x,y)|\right\}\leq Cd(x,y)^{\frac{2\alpha+1}{1-\alpha}+n}$$
      for all $x,y\in M$.
   \end{itemize} 
  \end{cor}
   \begin{proof}
   (i) Fix $x,y\in M$. Set 
   $$K(t,x,y):=H_{p,q}(t,x,y)-\sum_{k=1}^{b_{p,q}}\phi_k(x)\otimes \phi_k(y).$$
   By (i) and (ii) of Theorem \ref{thm: heat kernel}, Lemma \ref{lem:eigensection} and Corollary \ref{cor:estimate1}, there is a constant 
   $C:=C(\alpha,\beta,\gamma,\mu_1,K,|M|)>0$ such that
    $$|K(t,x,y)|\leq Ct^{\frac{2\alpha}{1-\alpha}},\ \forall t\geq 1,$$
    $$|K(t,x,y)|\leq Ct^{\frac{\alpha}{1-\alpha}}e^{-\frac{d(x,y)^2}{9t}},\ \forall t\in (0,1].$$
    Since $1<\alpha\leq 2$, then we get 
    $$\int_{0}^\infty s^{\frac{\alpha}{\alpha-1}-2}e^{-s}ds\leq \Gamma\left(\frac{1}{\alpha-1}\right),$$
    where $\Gamma(\cdot)$ is the Gamma function. By definition, we know 
   \begin{align*}
    |G_{p,q}(x,y)|&\leq \int_0^{\infty}|K(t,x,y)|dt\\
    &\leq C\int_0^1t^{\frac{\alpha}{1-\alpha}}e^{-\frac{d(x,y)^2}{9t}}dt+C\int_1^\infty t^{\frac{2\alpha}{1-\alpha}}dt\\
    &\leq 9^{\frac{1}{\alpha-1}}\Gamma\left(\frac{1}{\alpha-1}\right)Cd(x,y)^{\frac{2}{1-\alpha}}+C.
   \end{align*}
   (ii)   
   \begin{comment} In this case, it is well known the heat kernel of $M$ for functions satisfies 
    $$H(t,x,y)\leq Ct^{-n}e^{-\frac{d(x,y)^2}{9t}},\ \forall (t,x,y)\in \mr_+\times M\times M.$$
    By \cite[Theorem 4.3]{LX10}, we know 
    $$|H_{p,q}(t,x,y)|\leq e^{Kt}H(t,x,y),\ \forall (t,x,y)\in \mr_+\times M\times M,$$
    then we get 
    $$|H_{p,q}(t,x,y)|\leq Ce^K t^{-n}e^{-\frac{d(x,y)^2}{5t}},\ \forall (t,x,y)\in\mr_+\times M\times M.$$
    It is well known that $M$ satisfies the Sobolev-type inequality for $\alpha:=n/(n-1)$ if $n\geq 2$.
    By 
    \end{comment}
    Fix $x,y\in M$. We only give the estimate of $\bar\partial G_{p,q}(x,y)$.
     Since $M$ is compact, by Bishop-Gromov inequality,  we have 
    $$|B(y,\sqrt{t})|\leq b_ne^{(2n-1)\sqrt{\theta}\diam(M)}t^{n},\ \forall t>0,$$
    where $b_n$ is the volume of unit ball in $\mc^n$. In particular, we know 
    $$|M|\leq b_ne^{(2n-1)\sqrt{\theta}\diam(M)}\diam(M)^{2n}.$$
    By (iii) and (iv) of Theorem \ref{thm: heat kernel}, there is a constant 
   $$C:=C(\alpha,\beta,\gamma,r,n,\theta,\mu_1,K,K_+,K_{-},\diam(M))>0$$
    such that 
   $$|\bar\partial_y H_{p,q}(t,x,y)|\leq C t^{\frac{2\alpha}{1-\alpha}-\frac{1}{2}},\ \forall t\geq 1,$$
   $$|\bar\partial_y H_{p,q}(t,x,y)|\leq Ct^{\frac{4\alpha-1}{2-2\alpha}+\frac{n}{2}}e^{-\frac{d(x,y)^2}{36t}},\ \forall 0<t\leq 1.$$
   Clearly, 
   \begin{align*}
   |\bar\partial G_{p,q}(x,y)|&\leq \int_0^\infty |\bar\partial_y H_{p,q}(t,x,y)|dt\\
   &\leq C\int_0^{1}t^{\frac{4\alpha-1}{2-2\alpha}+\frac{n}{2}}e^{-\frac{d(x,y)^2}{36t}}dt+C\int_1^\infty t^{\frac{2\alpha}{1-\alpha}-\frac{1}{2}}dt\\
   &\leq C\Gamma\left(\frac{4\alpha-1}{2\alpha-2}-\frac{n}{2}-1\right)d(x,y)^{\frac{2\alpha+1}{1-\alpha}+n}+C.
   \end{align*}
   \end{proof}
   
  To prove Corollary \ref{cor:good sobolev}, let us firstly introduce some notations and give some lemmas. We adopt the same assumptions and notations as in Theorem \ref{thm: heat kernel}. Suppose $0<\mu_1\leq\mu_2<\cdots$ are all nonzero eigenvalues of 
   $\square_{p,q}$, and $\phi_{b_{p,q}+1},\phi_{b_{p,q}+2},\cdots$ are the corresponding eigensections such that 
   they consist of an orthonormal basis the orthogonal complement of $\mathcal{H}_{p,q}(M,E)$ in $L^2(M,\Lambda^{p,q}T^*M\otimes E)$.
   For any $f\in L^2(M,\Lambda^{p,q}T^*M\otimes E)$. Clearly, there are constants $a_k\in\mc$ such that
   $$f-\mathcal{P}f=\sum_{k=1}^\infty a_k\phi_{b_{p,q}+k},$$
   and then 
   $$\square_{p,q}f=\sum_{k=1}^\infty a_k\mu_k\phi_{b_{p,q}+k}.$$
   For any $\alpha\in\mr$, we define the pseudo-differential operator $\square_{p,q}^\alpha$ via 
   $$\square_{p,q}^\alpha f:=\sum_{k=1}^\infty a_k\mu_k^\alpha\phi_{b_{p,q}+k}.$$
   Using integration by parts, we have 
   \begin{equation}
    \int_{M}|\bar\partial f|^2+\int_{M}|\bar\partial^*f|^2=\int_{M}\langle f,\square_{p,q}f\rangle
    =\sum_{k=1}^\infty a_k^2\mu_k
    =\int_{M}|\square_{p,q}^{1/2}f|^2.
   \end{equation} 
   
   Set 
   $$K(t,x,y):=\left|H_{p,q}(t,x,y)-\sum_{k=1}^{b_{p,q}}\phi_k(x)\otimes \phi_k(y)\right|,\ \forall (t,x,y)\in \mr_+\times M\times M.$$
   For any $x\in M$, we have 
   \begin{equation}\label{equ:-1/2}
   \begin{split}
   &\quad \square_{p,q}^{-1/2}f(x)\\
   &=\sum_{k=1}^\infty a_k\mu_k^{-1/2}\phi_{b_{p,q}+k}(x)\\
   &=\sum_{k=1}^\infty\int_{M}\langle \phi_{b_{p,q}+k}(y),f(y)\rangle dV(y)\cdot \mu_k^{-1/2}\phi_{b_{p,q}+k}(x)\\
   &=\frac{1}{\Gamma(1/2)}\sum_{k=1}^\infty \int_0^\infty t^{-1/2} e^{-\mu_kt}dt\int_{M}\langle \phi_{b_{p,q}+k}(y),f(y)\rangle dV(y)\cdot \phi_{b_{p,q}+k}(x)\\
   &=\frac{1}{\Gamma(1/2)}\int_0^\infty t^{-1/2}\int_{M}\langle  K(t,x,y),f(y)\rangle dV(y) dt,
   \end{split}
   \end{equation}
   where $\Gamma(\cdot)$ is the Gamma function. One needs to note that $\Gamma(1/2)\geq 1$.
   
   \indent    A basic lemma is the following:
   \begin{lem}\label{lem:weak type}
   For any $1\leq \ell<(2\alpha)/(\alpha-1)$, set 
   $$-m:=\frac{1}{2}+\frac{\alpha}{(1-\alpha)\ell}<0,$$ 
   then $\square_{p,q}^{-1/2}$ is of weak $(\ell,\ell/(2m)+\ell)$ type, i.e.,
   there is a constant $C:=C(\alpha,\beta,\gamma,\mu_1,K,|M|)$ satisfying for any $\delta>0$ and  $g\in L^{\ell}(M,\Lambda^{p,q}T^*M\otimes E)$, we have 
   $$\left|\{x\in M|\ \square_{p,q}^{-1/2}g(x)>\delta\}\right|\leq C \delta^{-\frac{\ell}{2m}-\ell}\left(\int_{M}|g|^{\ell}\right)^{1+\frac{1}{2m}}.$$
   \end{lem}
   \begin{proof}
 The argument is inspired by the proof \cite[Theorem 11.6]{L12}. 
 By the proof (i) of Theorem \ref{thm: heat kernel}, we know 
 $$|H_{p,q}(t,x,y)|\leq e^{Kt}H(t,x,y),\ \forall (t,x,y)\in \mr_+\times M\times M.$$ 
 In the following, when we say a number is a constant, we mean it is a constant only depending on 
 $\alpha,\beta,\gamma,\mu_1,K,|M|$.
 By Corollary \ref{cor:estimate2}, we have 
 \begin{equation}\label{equ:bpq}
 b_{p,q}\leq r|M|(2\beta K+2\gamma)^{\frac{\alpha}{\alpha-1}}\cdot \alpha^{\frac{\alpha}{(\alpha-1)^2}}.
 \end{equation}
 Use Lemma \ref{lem:eigensection}, then there is a constant $C_1>0$ such that 
 \begin{equation}\label{equ:heat kernel bound0}
 \sup_M|\phi_k|^2\leq C_1,\ \forall 1\leq k\leq b_{p,q},
 \end{equation}
 then by the triangle inequality, we know 
 $$|K(t,x,y)|\leq C_1b_{p,q}+e^KH(t,x,y),\ \forall t\leq 1,$$
 which implies that for all $(t,x)\in(0,1]\times M$, we have 
 \begin{equation}\label{equ:heat kernel bound1}
  \int_{M}K(t,x,y)dV(y)\leq C_1b_{p,q}|M|+e^K.
 \end{equation}
 if one notes that 
 $$\sup_{x\in M}\int_{M}H(t,x,y)dV(y)\leq 1.$$
 Use (ii) of Theorem \ref{thm: heat kernel}, there is a constant $C_2>0$ such that 
 \begin{equation}
 K(t,x,y)\leq C_2t^{\frac{2\alpha}{1-\alpha}},\ \forall (t,x,y)\in\mr_+\times M\times M,
 \end{equation}
 then we get 
 \begin{equation}\label{equ:heat kernel bound2}
 \int_{M}K(t,x,y)dV(y)\leq C_2|M|,\ \forall (t,x)\in [1,\infty)\times M.
 \end{equation}
 By \eqref{equ:bpq}, \eqref{equ:heat kernel bound1} and \eqref{equ:heat kernel bound2}, there is a constant 
 $C_3>0$ such that 
 \begin{equation}\label{equ:heat kernel bound3}
 \sup_{\mr_+\times M}\int_{M}|K(t,x,y)|dV(y)\leq C_3.
 \end{equation}
 Use (i) of Theorem \ref{thm: heat kernel} and \eqref{equ:heat kernel bound0}, there is a constant 
 $C_4>0$ such that 
 \begin{equation}\label{equ:heat kernel bound4}
 |K(t,x,y)|\leq C_1b_{p,q}+C_4t^{\frac{\alpha}{1-\alpha}},\ \forall (t,x,y)\in (0,1]\times M\times M.
 \end{equation}
 Use (ii) of Theorem \ref{thm: heat kernel}, we know there is a constant $C_5>0$ such that 
 \begin{equation}\label{equ:heat kernel bound5}
 |K(t,x,y)|\leq C_5t^{\frac{\alpha}{1-\alpha}},\ \forall (t,x,y)\in[1,\infty)\times M\times M.
 \end{equation}
 Combining \eqref{equ:heat kernel bound4} and \eqref{equ:heat kernel bound5}, we conclude there is a constant 
 $C_6>0$ such that  
 \begin{equation}\label{equ:heat kernel bound6}
 |K(t,x,y)|\leq C_6t^{\frac{\alpha}{1-\alpha}},\ \forall (t,x,y)\in(0,\infty)\times M\times M.
 \end{equation}
 \indent Fix $1\leq\ell<(2\alpha)/(\alpha-1)$, and fix $g\in L^{\ell}(M,\Lambda^{p,q}T^*M\otimes E).$
 We may assume $\ell>1$ since the case $\ell=1$ can be proved similarly.
  % Thus, we only need to show that 
   %\begin{equation}
   %\left(\int_{M}|f|\right)\leq \int_{M}|\square_{p,q}^{1/2}f|^2.
   %\end{equation}
   %\indent Set 
   %$$g:=\square_{p,q}^{1/2}f,$$
  % then  we only need to show that 
   %\begin{equation}
   %\left(\int_{M}|\square_{p,q}^{-1/2}g|\right)\leq \int_{M}|g|^2.
   %\end{equation}
  % By Marcinkiewicz interpolation theorem, we only need to show $\square_{p,q}^{-1/2}$ is of weak $(\ell,)$ type,
   %i.e. for any $\beta>0
   For $T>0$ which will be determined later, let us define 
   $$T_1g(x):=\frac{1}{\Gamma(1/2)}\int_0^{T}t^{-1/2}\int_{M}\langle  K(t,x,y),g(y)\rangle dV(y) dt,\ \forall x\in M,$$
   $$T_2g(x):=\frac{1}{\Gamma(1/2)}\int_{T}^{\infty}t^{-1/2}\int_{M}\langle  K(t,x,y),g(y)\rangle dV(y) dt,\ \forall x\in M.$$
   By \eqref{equ:-1/2}, we have 
   $$\square_{p,q}^{-1/2}g=T_1g+T_2g.$$
   For  any $x\in M$, by Cauchy-Schwarz inequality, H\"older's inequality and \eqref{equ:heat kernel bound3},
   we have 
   \begin{align*}
   &\quad |T_1g(x)|^\ell\\
   &\leq \left(\int_0^{T}t^{-1/2}\int_{M}|K(t,x,y)|\cdot|g(y)|dV(y) dt\right)^\ell\\
   &\leq \left(\int_0^Ts^{-1/2}ds\right)^{\ell-1}\int_0^Tt^{-1/2}\left(\int_{M}|K(t,x,y)|\cdot|g(y)|dV(y)\right)^{\ell}dt\\
   &\leq (2\sqrt{T})^{\ell-1}C_3^{\ell-1}\int_0^Tt^{-1/2}\int_{M}|K(t,x,y)|\cdot |g(y)|^\ell dV(y)dt,
   \end{align*}
   then by Fubini's theorem (and note that $|K(t,x,y)|=|K(t,y,x)|$), we get 
   \begin{equation}\label{equ:estimate T1}
   \int_{M}|T_1g(x)|^{\ell}dV(x)\leq  (2\sqrt{T})^{\ell}C_3\int_{M}|g(y)|^{\ell}dV(y).
   \end{equation}
   Similarly, for any $x\in M$, by H\"older's inequality, \eqref{equ:heat kernel bound3} and \eqref{equ:heat kernel bound6}, we know 
   \begin{align}\label{equ:estimate T2}
   &\quad |T_2g(x)|\\
   &\leq \int_T^{\infty}t^{-1/2}\int_{M}|K(t,x,y)|\cdot|g(y)|dV(y) dt\nonumber\\
   &\leq \int_T^{\infty}t^{-1/2}\left(\int_{M}|K(t,x,y)|^{\frac{\ell}{\ell-1}}dV(y)\right)^{\frac{\ell-1}{\ell}}
   \cdot \left(\int_{M}|g|^{\ell}dV\right)^{1/\ell}\nonumber\\
   &\leq C_3^{\frac{\ell-1}{\ell}}C_6^{\frac{1}{\ell}}\int_T^{\infty}t^{-\frac{1}{2}+\frac{\alpha}{(1-\alpha)\ell}}dt\cdot \left(\int_{M}|g|^{\ell}dV\right)^{1/\ell}\nonumber\\
   &=\frac{C_3^{\frac{\ell-1}{\ell}}C_6^{\frac{1}{\ell}}}{\left|\frac{1}{2}+\frac{\alpha}{(1-\alpha)\ell}\right|}T^{-m}\left(\int_{M}|g|^{\ell}dV\right)^{1/\ell}.
   \nonumber
   \end{align}
  For any $\delta>0$, by choosing $T$ such that 
  $$\frac{C_3^{\frac{\ell-1}{\ell}}C_6^{\frac{1}{\ell}}}{m}T^{-m}\left(\int_{M}|g|^{\ell}dV\right)^{1/\ell}
  =\frac{\delta}{2},$$
  i.e. 
  $$T=\left(\frac{2C_3^{\frac{\ell-1}{\ell}}C_6^{\frac{1}{\ell}}}{m}\right)^{1/m}\delta^{-1/m}\left(\int_{M}|g|^{\ell}dV\right)^{\frac{1}{m\ell}},$$
  then by \eqref{equ:estimate T1} and \eqref{equ:estimate T2}, we have 
  \begin{align}
   &\quad \left|\{x\in M|\ |\square_{p,q}^{-1/2}g(x)|>\delta\}\right|\leq \left|\{x\in M|\ |T_1g(x)|>\delta/2\}\right|\\
   &\leq \left(\frac{2}{\delta}\right)^\ell\left(\int_{M}|T_1g|^{\ell}\right)\leq C_3\left(\frac{4\sqrt{T}}{\delta}\right)^{\ell}\left(\int_{M}|g|^{\ell}\right)\nonumber\\
   &= 4^{\ell}\left(\frac{2C_3^{\frac{\ell-1}{\ell}}C_6^{\frac{1}{\ell}}}{m}\right)^{\frac{\ell}{2m}}\delta^{-\frac{\ell}{2m}-\ell}\left(\int_{M}|g|^{\ell}\right)^{1+\frac{1}{2m}},\nonumber
  \end{align}
  which shows that $\square_{p,q}^{-1/2}$ is of weak $(\ell,\ell/(2m)+\ell)$ type. 
 \end{proof}
 Use Marcikiewicz interpolation theorem (see \cite[Lemma 11.5]{L12})  as shown in the proof of \cite[Theorem 11.6]{L12}, the following Corollary can be easily derived from  Lemma \ref{lem:weak type}.
 \begin{cor}\label{cor:strong type}
 For any $1\leq \ell<(2\alpha)/(\alpha-1)$, set 
 $$-m:=\frac{1}{2}+\frac{\alpha}{(1-\alpha)\ell},$$
 and suppose $1\leq k<\infty$ satisfy
 $$
 \begin{cases}
                             k<\ell/(2m)+\ell, & \mbox{if } \ell=1, \\
                             k\leq \ell/(2m)+\ell, & \mbox{if }  1<\ell<(2\alpha)/(\alpha-1), \\
    \end{cases}
 $$
 then there is a constant $C:=(k,\ell,\alpha,\beta,\gamma,\mu_1,K,|M|)>0$ such that 
 $$\left(\int_{M}\left|\square_{p,q}^{-1/2}f\right|^{k}\right)^{1/k}
 \leq C\left(\int_{M}\left|f\right|^{\ell}\right)^{\frac{1}{\ell}},\ \forall f\in L^{\ell}(M,\Lambda^{p,q}T^*M\otimes E).$$
 \end{cor}
 \indent Now we can get the proof of Corollary \ref{cor:good sobolev}

\begin{cor}[= Corollary \ref{cor:good sobolev}]\label{cor:good proof}
 Under the same notations and assumptions of Theorem  \ref{thm: heat kernel}. 
 For any $1\leq \ell<(2\alpha)/(\alpha-1)$, set 
 $$-m:=\frac{1}{2}+\frac{\alpha}{(1-\alpha)\ell},$$
 and suppose $1\leq k<\infty$ satisfy
 $$
 \begin{cases}
                             k<\ell/(2m)+\ell, & \mbox{if } \ell=1, \\
                             k\leq \ell/(2m)+\ell, & \mbox{if }  1<\ell<(2\alpha)/(\alpha-1). \\
    \end{cases}
 $$ 
 Then there is a constant  $C:=C(k,\ell,\alpha,\beta,\gamma,r,\mu_1,K,|M|)>0$
 such that the following inequality holds:
 $$\left(\int_{M}\left|f-\mathcal{P}f\right|^{k}\right)^{1/k}
 \leq C\left(\int_{M}|\square_{p,q}f|^{\ell}\right)^{1/\ell},\ \forall f\in C^1(M,\Lambda^{p,q}T^*M\otimes E),$$
 where $\mathcal{P}\colon L^2(M,\Lambda^{p,q}T^*M\otimes E)\rw \mathcal{H}_{p,q}(M,E)$ denotes the Bergman projection.
 \end{cor}
\begin{proof}
 Let $\mathcal{G}_{p,q}$ be  the Green operator for $\square_{p,q}$. 
 By a density argument, we may assume $f$ is smooth. 
 Set $g:=f-\mathcal{P}f$, and set $h:=\square_{p,q}g$, then we know 
 $$\mathcal{G}_{p,q}h=f-\mathcal{P}f,\ h=\square_{p,q}f.$$
 By definition, it is obvious that 
$$\mathcal{G}_{p,q}h=\square_{p,q}^{-1}h=\square_{p,q}^{-1/2}\circ \square_{p,q}^{-1/2}h.$$
By Corollary \ref{cor:strong type}, there $C:=C(k,\ell,\alpha,\beta,\gamma,r,\mu_1,K,|M|)>0$
 such that 
 $$\left(\int_{M}\left|\mathcal{G}_{p,q}h\right|^{k}\right)^{1/k}\leq C\left(\int_{M}|h|^{\ell}\right)^{1/\ell},$$ 
then the conclusion follows. 
\end{proof} 
\section{$L^{q,p}$-estimates of $\bar\partial$}\label{sec:estimate}
In this section, we will give the proof of Theorem \ref{thm:complex L^p estimate}.
\begin{thm}[= Theorem \ref{thm:complex L^p estimate}]
    Under the same assumptions and notations of Corollary \ref{cor:dbar estimate} and we further assume 
    $s>1$. Then there is a constant 
    $$C:=C(k,s,r,n,\mu_1,\theta,K_{-},\diam(M))>0$$ 
    such that 
    for any $\bar\partial$-closed $f\in L^s(M,\Lambda^{p,q}T^*M\otimes E)$, we may find a $u\in L^k(M,\Lambda^{p,q-1}T^*M\otimes E)$ satisfying
    $$\bar\partial u=f,\ \left(\int_{M}\left|u\right|^{k}dV\right)^{\frac{1}{k}}
 \leq C\left(\int_{M}\left|f\right|^{s}dV\right)^{\frac{1}{s}}.$$ 
  Furthermore, when $s=2$, the constant  $C$ can be chosen to depend only on 
  $k,r,n,\mu_1,|M|$, and the Sobolev constants $\beta,\gamma$ when $\alpha=n/(n-1).$
 \end{thm}
 \begin{proof}
 By Corollary \ref{cor:vanish}, we know $\mathcal{H}_{p,q}(M,E)=0$.
 Fix a $\bar\partial$-closed $f\in L^s(M,\Lambda^{p,q}T^*M\otimes E),$ then we may assume $f$ is smooth by density.
 Set  
 $$u(x):=\int_{M}\langle f(y),\bar\partial_y^* G_{p,q}(x,y)\rangle dV(y),$$
 where $G_{p,q}$ is the Green form. By Hodge theory and the fact $\mathcal{H}_{p,q}(M,E)=0$, we know 
 $\bar\partial u=f.$ The estimate for $\bar\partial^*G$ from 
 Corollary \ref{cor:green form} (whose constant, by tracing the proof of Corollary \ref{cor:green form}, can be made independent of 
  $K_+$), combined with Young's convolution inequality or Schur's test (as in \cite{DHQ24}),  yields the desired $L^{k,s}$-estimates.
  The last claim follows directly from Corollary \ref{cor:good sobolev}.
 \end{proof}
\section{Sobolev-type inequalities and eigenvalues estimates for K\"ahler families}\label{sec:family}
\indent In this section, we will give a precise statement and provide a more general version of Theorem \ref{thm:family version}
 within the framework of \cite{GPSS24}.
To do this, we firstly recall the set-up of \cite{GPSS24} and state some lemms. \\
\indent Let \((X, \omega_0)\) be an \(n\)-dimensional compact K\"ahler manifold equipped with a K\"ahler metric \(\omega_0\). Let \(\mathcal{K}(X)\) be the space of K\"ahler metrics on \(X\) and define the \(p\)-th Nash-Yau entropy of a K\"ahler metric \(\omega \in \mathcal{K}(X)\) associated to \((X, \omega_0)\) via
\[
V_\omega:= \int_{X} \omega^n,\ \mathcal{N}_{X, \omega_0, m}(\omega) := \frac{1}{V_\omega} \int_{X} \left| \log \left( (V_\omega)^{-1} \frac{\omega^n}{\omega_0^n} \right) \right|^m \omega^n, \quad ,
\]
for \(m > 0\). For given parameters $A,B>0$, $m>n$, define the following set of admissible metrics: 
\[
\mathcal{V}(X, \omega_0, n,m,A, B) := \left\{ \omega \in \mathcal{K}(X)\right.\left|\ I_\omega\leq A,\ \mathcal{N}_{X, \omega_0, m}(\omega) \leq B \right\},
\]
where $I_\omega:=[\omega]\cdot[\omega_0]^{n-1}$ is the intersection number of the K\"ahler classes $[\omega]$ and $[\omega_0]$.\\
\indent By the works of \cite{GPSS22} and \cite{GPSS24}, we have the following Sobolev-type inequality:
\begin{lem}\label{lem:sobolev family functions}
Given $m>n,\ A>0,\ B>0$, there exist $\alpha:=\alpha(n,m)>1$ and $C:=C(X,\omega_0,n,m,\alpha,A,B)>0$ such that for any K\"ahler metric $\omega\in \mathcal{V}(X, \omega_0, n,m,A,B)$
and any $f\in C^1(X)$, the following inequality holds:
$$\left(\frac{1}{V_\omega}\int_{X}|f|^{2\alpha}\omega^n\right)^{1/\alpha}\leq C
\left(\frac{I_\omega}{V_\omega}\int_{X}|\nabla f|^2\omega^n+\frac{1}{V_\omega}\int_{X}|f|^2\omega^n\right).$$
\end{lem}
\indent For any K\"ahler metric $\omega$ on $X$, let $H_{\omega}$ be the heat kernel of $(X,\omega)$ for functions.
Fix $p,q\in\{0,\cdots,n\}$. Let $E$ be a Hermitian holomorphic vector bundle of rank $r$ over $X$, and let $H_{p,q,\omega}$ 
be the heat kernel of $(X,\omega)$ with respect to the $\bar\partial$-Laplacian $\square_{p,q,\omega}$ which correpsonds
to the metric $\omega$. By adopting similar notations as in \ref{thm: heat kernel}, there are  corresponding distance functions $d_{\omega}(\cdot,\cdot)$,
Weitzenb\"ock curvature operator $\mathfrak{Ric}_{p,q,\omega}^E$,
Hodge numbers $b_{p,q,\omega}$, first nonzero eigenvalues $\mu_{1,\omega}$, eigensections $\phi_{1,\omega},\cdots,\phi_{b_{p,q,\omega},\omega}$
and Bergman projections $\mathcal{P}_\omega$.
For our purpose, we introducing the following set of K\"ahler metrics:
For any $A>0,\  B>0,\ K\geq 0,\ m>n$, define 
\[
\mathcal{V}_{p,q}^E(X, \omega_0, n,m,A, B,K) := \left\{ \omega \in \mathcal{V}(X, \omega_0, n,m,A, B)\right.\left|\ 
\mathfrak{Ric}_{p,q,\omega}^E\geq -K\right\}.
\]
Combining Theorem \ref{thm: heat kernel} and Lemma \ref{lem:sobolev family functions}, one can conclude the following:
\begin{prop}\label{prop:family}
Given $m>n,\ A>0,\ B>0,\ K\geq 0$, there exist $\alpha:=\alpha(n,m)>1$ and $C:=C(X,\omega_0,n,m,\alpha,A,B)>0$ such that 
for any K\"ahler metric $\omega\in \mathcal{V}_{p,q}^E(X, \omega_0, n,m,A,B,K)$, we have the following estimates of heat kernels:
     $$|H_{p,q,\omega}(t,x,y)|\leq \frac{C}{V_\omega}\left(\frac{I_\omega}{t}\right)^{\frac{\alpha}{\alpha-1}}e^{Kt+I_\omega^{-1}t}e^{-\frac{d_\omega(x,y)^2}{9t}},$$
$$\left|H_{p,q,\omega}(t,x,y)-\sum_{k=1}^{b_{p,q,\omega}}\phi_{k,\omega}(x)\otimes \phi_{k,\omega}(y)\right|\leq \frac{Cr}{V_\omega}\left(\frac{I_\omega}{t}\right)^{\frac{2\alpha}{\alpha-1}}e^{\frac{2(KI_\omega+1)\alpha}{(\alpha-1)I_\omega\mu_{1,\omega}}},$$
for all $(t,x,y)\in\mr_+\times X\times X$.
\end{prop}
By Lemma \ref{lem:eigensection} and Proposition \ref{prop:family}, one may modified the proof of Lemma \ref{lem:weak type} a little
to conclude the following result:
\begin{prop}\label{prop:family weak type}
Given $m>n,\ A>0,\ B>0,\ K\geq 0$, choose an $\alpha$ as in Propoistion \ref{prop:family}. For any $1\leq \ell<(2\alpha)/(\alpha-1)$, set 
$$-\rho:=\frac{1}{2}+\frac{\alpha}{(1-\alpha)\ell}<0,$$ 
   then there is a constant $C:=C(X,\omega_0,n,m,\ell,\alpha,A,B)$ such that for all $\delta>0$, 
    any $\omega\in \mathcal{V}_{p,q}^E(X, \omega_0, n,m,A,B),$ and any $L^{\ell}$-integrable $E$-valued $(p,q)$-form 
   $g$ on $(X,\omega)$, we have 
   $$\left|\{x\in X|\ \square_{p,q,\omega}^{-1/2}g(x)>\delta\}\right|\leq C\left(\frac{\eta^\ell}{V_\omega}\right)^{\frac{1}{2\rho}}\delta^{-\frac{\ell}{2\rho}-\ell}\left(\int_{X}|g|^{\ell}\omega^n\right)^{1+\frac{1}{2\rho}},$$
   where 
   $$\eta:=2r(I_\omega K+1)^{\frac{2\alpha}{\alpha-1}}+e^K
   +I_\omega^{\frac{2\alpha}{\alpha-1}}e^{K+I_\omega^{-1}}+rI_\omega^{\frac{2\alpha}{\alpha-1}}e^{\frac{2(KI_\omega+1)\alpha}{(\alpha-1)I_\omega\mu_{1,\omega}}}.$$
  % $$\eta_2:=\max\left\{r(I_\omega K+1)^{\frac{2\alpha}{\alpha-1}}+I_\omega^{\frac{\alpha}{\alpha-1}}e^{K+I_\omega^{-1}}
   %,rI_\omega^{\frac{2\alpha}{\alpha-1}}e^{\frac{KI_\omega+1}{I_\omega\mu_{1,\omega}}}\right\}.$$
\end{prop} 
\begin{proof}
Fix a $\omega\in\mathcal{V}_{p,q}^E(X, \omega_0, n,m,A,B)$, then the constants $C_1,\cdots,C_6$ in the proof of Lemma \ref{lem:weak type} can be taken to be 
$$C_1=\frac{C}{V_\omega}(I_\omega K+1)^{\frac{\alpha}{\alpha-1}},\ C_2=C_5=\frac{Cr}{V_\omega}I_\omega^{\frac{2\alpha}{\alpha-1}}e^{\frac{2(KI_\omega+1)\alpha}{(\alpha-1)I_\omega\mu_{1,\omega}}},$$
$$ C_3:=\max\{C_1b_{p,q,\omega}V_\omega+e^K,C_2V_\omega\},\  C_4=\frac{C}{V_\omega}I_\omega^{\frac{\alpha}{\alpha-1}}e^{K+I_\omega^{-1}},$$
$$ C_6:=\max\{C_1b_{p,q,\omega}+C_4,C_5\},$$
where by Corollary \ref{cor:estimate1}, 
$$b_{p,q,\omega}\leq Cr(I_\omega K+1)^{\frac{\alpha}{\alpha-1}},$$
and $C$ is a constant only depending on $X,\omega_0,n,m,\alpha,A,B$,

\end{proof}
\begin{comment}
 One needs to note that if 
$$\ell<\frac{1}{2}\left(2+\frac{2\alpha}{\alpha-1}\right),$$
then 
$$\frac{1}{2\rho}\leq 2\alpha,$$
and thus 
$$\max\{\eta_1,\eta_2\}^{1/(2\rho)}\leq \max\left\{1,\max\{\eta_1,\eta_2\}^{2\alpha}\right\}.$$
\end{comment}
Use Marcikiewicz interpolation theorem (see \cite[Theorem 1.4.3]{H20} and its proof), the following theorem can be concluded from 
Proposition \ref{prop:family weak type}.
\begin{thm}\label{prop:family strong type}
Under the same assumptions and notations of Proposition \ref{prop:family weak type}.
Suppose $1\leq k<\infty$ satisfy
 $$
 \begin{cases}
                             k<\ell/(2\rho)+\ell, & \mbox{if } \ell=1, \\
                             k\leq \ell/(2\rho)+\ell, & \mbox{if }  1<\ell<(2\alpha)/(\alpha-1), \\
    \end{cases}
 $$
 then there is a constant $C:=C(X,\omega_0,n,m,k,\ell,\alpha,A,B)>0$ such that 
 for any $\omega\in \mathcal{V}_{p,q}^E(X, \omega_0, n,m,A,B)$ and any $L^{\ell}$-integrable $E$-valued $(p,q)$-form    $f$ on $(X,\omega),$ we have 
 $$\left(\int_{X}\left|\square_{p,q}^{-1/2}f\right|^{k}\omega^n\right)^{1/k}
 \leq CV_\omega^{\frac{1}{k}-\frac{1}{\ell}}(1+\eta)\left(\int_{X}\left|f\right|^{\ell}\omega^n\right)^{\frac{1}{\ell}},$$
 where $\eta$ is defined in Proposition \ref{prop:family weak type}. 
\end{thm}
Now we give a more general version of the first part of Theorem \ref{thm:family version}, which follows from 
Theorem \ref{prop:family strong type} by using the same ideas as in the proof of Corollary \ref{cor:good proof},
and so we omit the details.
\begin{thm}\label{thm:sobolev family good}
Given $m>n,\ A>0,\ B>0,\ K\geq 0.$ 
Choose $\alpha$ as in Proposition \ref{prop:family}.
For any 
$$1\leq \ell<\frac{1}{2}\left(2+\frac{2\alpha}{\alpha-1}\right),$$
set 
$$-\rho:=\frac{1}{2}+\frac{\alpha}{(1-\alpha)\ell}<0.$$
Suppose $1\leq k<\infty$ satisfy
 $$
 \begin{cases}
                             k<\ell/(2\rho)+\ell, & \mbox{if } \ell=1, \\
                             k\leq \ell/(2\rho)+\ell, & \mbox{if }  \ell>1, \\
    \end{cases}
 $$
 then there is a constant $C:=C(X,\omega_0,n,m,k,\ell,\alpha,A,B)>0$ such that 
 for any  K\"ahler metric $\omega\in \mathcal{V}_{p,q}^E(X, \omega_0, n,m,A,B)$ and any $f\in C^2(X)$, the following Sobolev-type inequality holds:
$$\left(\int_{X}\left|f-\mathcal{P}_\omega f\right|^{k}\right)^{1/k}\leq CV_\omega^{\frac{1}{k}-\frac{1}{\ell}}
(1+\eta)\left(\int_{X}|\square_{p,q}f|^{\ell}\right)^{1/\ell},$$
where $\eta$ is defined in Proposition \ref{prop:family weak type}. 
\end{thm}
In the last, let us give a lower bound for the first eigenvalues of $\square_{p,q,\omega}$ under some stronger assumptions,
which clearly implies the second part of Theorem \ref{thm:family version}.
For any K\"ahler metric $\omega$ on $X$, let  $0<\mu_{1,\omega}\leq \mu_{2,\omega}<\cdots$ be all nonzero eigenvalues of $\square_{p,q,\omega}$.
\begin{prop}\label{prop:first eigenvalue family}
Given $m>n,\ A>0,\ B>0$, then there exists  a constant $c:=c(X,\omega_0,n,m,A,B)>0$
such that for any $\omega\in \mathcal{V}_{p,q}^E(X, \omega_0, n,m,A,B,0)$ satisfying $b_{p,q,\omega}\geq 1$, we have 
$$\mu_{1,\omega}\geq \frac{c}{I_\omega}.$$
\end{prop}
\begin{proof}
By \cite[Theorem 2.2]{GPSS23} and \cite[Theorem 4.3]{LX10},
there are constants $C:=C(X,\omega_0,n,m,A,B)>0,\ c:=c(X,\omega_0,n,m,A,B)>0$ such that for any 
$\omega\in \mathcal{V}_{p,q}^E(X, \omega_0, n,m,A,B,0)$, we have 
$$|H_{p,q,\omega}(t,x,x)|\leq \frac{1}{V_\omega}\left(1+Ce^{-cI_{\omega}^{-1}t}\right),\ \forall t\geq I_\omega.$$
Taking traces on both sides, then we get 
$$b_{p,q,\omega}+e^{-\mu_{1,\omega} t}\leq 1+Ce^{-cI_{\omega}^{-1}t},\ \forall t\geq I_\omega,$$
then by assumptions, we know 
$$e^{-\mu_{1,\omega}t}\leq Ce^{-cI_{\omega}^{-1}t}.$$
Let $t\rw \infty$, then we get 
$$\mu_{1,\omega}\geq \frac{c}{I_\omega}.$$
\end{proof}
Similar to the proof of \eqref{equ:eigenvalues estimates}, by Proposition \ref{prop:family}
and \ref{prop:first eigenvalue family}, one can easily get the following estimates of eigenvalues: 
\begin{thm}\label{thm:eigenvalue family}
Given $m>n,\ A>0,\ B>0$, then there are constants $\alpha:=\alpha(n,m)>1$ and  $c:=c(X,\omega_0,n,m,A,B)>0$
such that for any $\omega\in \mathcal{V}_{p,q}^E(X, \omega_0, n,m,A,B,0)$ satisfying $b_{p,q,\omega}\geq 1$, we have 
$$\mu_{k,\omega}\geq \frac{c}{I_\omega}\left(\frac{k}{r}\right)^{\frac{\alpha-1}{\alpha}},\ \forall k\geq 1.$$
\end{thm}
\indent Clearly, Theorem \ref{thm:eigenvalue family} is a generalization of \cite[Corollary 2.1]{GPSS23}.

 \bibliographystyle{alphanumeric}

\begin{thebibliography}{123456789}

%\bibitem[A82]{A82} T. Aubin, \emph{ Nonlinear analysis on manifolds, Monge-Amp\'ere equations}, Grundlehren Math. Wiss., 252, Springer-Verlag, New York, 1982, xii+204 pp.

\bibitem[B02]{B02} B. Berndtsson, \emph{ An eigenvalue estimate for the $\bar\partial$-Laplacian}, J. Differential Geom. {\bf 60} (2002), no. {\bf 2}, 295–313.

 \bibitem[BDG23]{BDG23}  R. Baumgarth, B. Devyver and B. Güneysu, \emph{ Estimates for the covariant derivative of the heat semigroup on differential forms, and covariant Riesz transforms}, Math. Ann. {\bf 386} (2023), no. {\bf 3-4}, 1753–1798.

%\bibitem[CL98]{CL98} Y. Chen, and L. Wu, \emph{ Second order elliptic equations and elliptic systems}, Transl. Math. Monogr., 174, American Mathematical Society, Providence, RI, 1998, xiv+246 pp.
    
%\bibitem[CL81]{CL81} S. Y. Cheng and P. Li, \emph{ Heat kernel estimates and lower bound of eigenvalues}, Comment. Math. Helvetici {\bf 56} (1981) 327–338.

%\bibitem[CLY81]{CLY81} S. Y. Cheng, P. Li and Yau, S. T, \emph{ On the upper estimate of the heat kernel of a complete Riemannian manifold}, Amer. J. Math. {\bf 103} (1981), 1021–1063.

%\bibitem[CM16]{CM16} A. Cianchi and V. Maz'ya, \emph{ Sobolev inequalities in arbitrary domains},  Adv. Math.{\bf 293}(2016), 644-696.

 \bibitem[D90]{D90} E.B. Davies, \emph{ Heat kernels and spectral theory}, Cambridge Tracts in Mathematics, 92. Cambridge University Press, Cambridge, 1990.
   
   
 % \bibitem[D12]{D12} J. P. Demailly, \emph{ Complex analytic and differential geometry}, electric book, Version of Thursday June 21,2012, available on the author's homepage.
 
    
\bibitem[DHQ24]{DHQ24} F. Deng, G. Huang, and X. Qin, \emph{ Uniform estimates of Green's functions and Sobolev-type inequalities on real and complex manifolds}, Preprint, Arxiv: 2409.19353 (2024).
    
\bibitem[DHQ25]{DHQ25} F. Deng, G. Huang, and X. Qin. \emph{ Some Sobolev-type inequalities for twisted differential forms on real and complex manifolds}, Preprint, Arxiv: 2501.05697 (2025).

%\bibitem[DJQ24]{DJQ24} F. Deng, W. Jiang, and X. Qin, \emph{ $\bar{\partial}$ Sobolev-type inequality and an improved $L^2$-estimate
%of $\bar{\partial}$ on bounded strictly pseudoconvex domains}, Preprint, Arxiv: 2401.15597 (2024).

\bibitem[DL82]{DL82} H. Donnelly and  P. Li, \emph { Lower bounds for the eigenvalues of Riemannian manifolds}, Michigan Math. J.{\bf 29}(1982), no.{\bf 2}, 149–161.
   
%\bibitem[F95]{F95} G. B. Folland, \emph{Introduction to partial differential equations},  Second edition, Princeton University Press, Princeton, NJ, 1995. xii+324 pp.

  \bibitem[EGHP23]{EGHP23} M. Egidi,  K. Gittins, G. Habib, and N. Peyerimhoff, \emph{ Eigenvalue estimates for the magnetic Hodge Laplacian on differential forms},
     J. Spectr. Theory   {\bf13} (2023), no. {\bf 4}, 1297–1343.
     
%  \bibitem[FGS86]{FGS86} E.B. Fabes, N. Garofalo, and S. Salsa, \emph{ A backward Harnack inequality and Fatou theorem for nonnegative solutions of parabolic equations}, Illinois J. Math. 30 (1986) 536–565
 
   \bibitem[GM75]{GM75} S. Gallot and D. Meyer, \emph{ Opérateur de courbure et laplacien des formes différentielles d’une variété riemannienne}, J. Math. Pures Appl. (9) {\bf 54} (1975), no. {\bf 3}, 259–284.
 
\bibitem[GPS24]{GPS24} B. Guo, D.H. Phong, and J. Sturm. \emph{ Green's functions and complex Monge-Ampère equations}, J. Differential Geom. {\bf 127} (2024),
no. {\bf 3}, 1083-1119.

\bibitem[GPSS22]{GPSS22} B. Guo, D.H. Phong, J. Song and J. Sturm, \emph{ Diameter estimates in Kähler geometry}, Preprint, Arxiv: 2209.09428 (2022).

\bibitem[GPSS23]{GPSS23} B. Guo, D.H. Phong, J. Song and J. Sturm, \emph{ Sobolev inequalities on Kähler spaces}, Preprint, Arxiv: 2311.00221 (2023).

\bibitem[GPSS24]{GPSS24} B. Guo, D.H. Phong, J. Song and J. Sturm, \emph{ Diameter estimates in Kähler geometry II: removing the small degeneracy assumption},
 Math. Z. {\bf 308} (2024), no. {\bf 3}, Paper No. 43, 7 pp.

%\bibitem[GPSS24]{GPSS24} B. Guo, Bin, D.H. Phong, J.Song, J. Sturm, 
%Diameter estimates in Kähler geometry II: removing the small degeneracy assumption.(English summary)

%\bibitem[GT01]{GT01} D. Gilbarg and N.S. Trudinger, \emph{ Elliptic partial differential equations of second order}, Reprint of the 1998 edition Classics Math., Springer-Verlag, Berlin, 2001. xiv+517 pp.
    
\bibitem[GT06]{GT06} V. Gol'dshtein and M. Troyanov, \emph{ Sobolev inequalities for differential forms and $L_{q,p}$-cohomology}, J. Geom. Anal.  {\bf 16} (2006), no. {\bf4}, 597–631.

  
\bibitem[GT24]{GT24} V. Guedj and T. Tô, \emph{ Kähler families of Green's functions}, Preprint, Arxiv: 2405.17232 (2024).

%\bibitem[GW82]{GW82} M. Gr\"uter and  K. O. Widman, \emph{ The Green function for uniformly elliptic equations}, Manuscripta Math.{\bf 37}(1982), no.{\bf 3}, 303–342.
 
\bibitem[H20]{H20} C. Hao, \emph{ Lecture notes on harmonic analysis}, Available online at \url{http://www.math.ac.cn/kyry/hcc/teach/201912/P020200428625504254918.pdf}.
      
%\bibitem[H67]{H67} L. H\"ormander, \emph{ $L^p$-estimates for (pluri-) subharmonic functions}, Math. Scand. {\bf 20}(1967), 65–78.
          
%\bibitem[H99]{H99} E. Hebey, \emph{Nonlinear Analysis on Manifolds: Sobolev Spaces and Inequalities}, Courant Lect. Notes Math., 5, New York University, Courant Institute of Mathematical Sciences, New York; American Mathematical Society, Providence, RI, 1999. x+309 pp.
  
%\bibitem[K01]{K01} S. G. Krantz, \emph{ Function Theory of Several Complex Variables}, Reprint of the 1992 edition, AMS Chelsea Publishing, Providence, RI, 2001, xvi+564 pp.


%\bibitem[L92]{L92} S. Laurent, \emph{ Uniformly elliptic operators on Riemannian manifolds}, J. Differential Geom. {\bf 36} (1992), no. {\bf 2}, 417–450.
       
%\bibitem{K05} G. Krantz, \emph{calculation and estimation of the Poisson kernel},J. Math. Anal. Appl, 302(2005), 143-148.

%\bibitem[L09]{L09} X. Li, \emph{ On the strong $L^p$-Hodge decomposition over complete Riemannian manifolds}, J. Funct. Anal. {\bf 257} (2009), no. {\bf 11}, 3617-3646.
        
%\bibitem[L10]{L10} X. Li, \emph{ Riesz transforms on forms and $L^p$-Hodge decomposition on complete Riemannian manifolds}, Rev. Mat. Iberoam. {\bf 26} (2010), no. {\bf 2}, 481-528.

  \bibitem[L12]{L12} P. Li, \emph{ Geometric analysis}, Cambridge Stud. Adv. Math., 134, Cambridge University Press, Cambridge, 2012, x+406 pp.

\bibitem[LX10]{LX10} X. Li, \emph{ $L^{p}$-estimates and existence theorems for the $\bar{\partial}$-operator on complete Riemannian manifolds}, Adv.Math. 224(2010),no.2,620-247.

\bibitem[LZZ21]{LZZ21} Z. Lu, Q. Zhang and Meng Zhu, \emph{ Gradient and eigenvalue estimates on the canonical bundle on K\"ahler manifolds}, J. Geom. Anal. {\bf 31}(2021), no. {\bf 10}, 10304-10335.

%\bibitem[MV05]{MV05} F. Maggi and C. Villani, \emph{Balls have the worst best Sobolev inequalities}, J. Geom. Anal. {\bf 15}(2005), no.{\bf 1}, 83-121.
  
 %\bibitem[N21]{N21} L.I. Nicolaescu, \emph{ Lectures on the geometry of manifolds}, Third edition, World Scientific Publishing Co. Pte. Ltd., Hackensack, NJ, [2021], \copyright 2021. xviii+682 pp.


% \bibitem[R10]{R10} F. Robert, \emph{ Existence et asymptotiques optimales des fonctions de Green des opérateurs elliptiques d’ordre deux},
 %       Available online at {\url{http://www.iecn.u-nancy.fr/~frobert/ConstrucGreen.pdf}}.

 \bibitem[RY94]{RY94} S. Rosenberg and D. Yang, \emph{ Bounds on the fundamental group of a manifold with almost nonnegative Ricci curvature},
J. Math. Soc. Japan {\bf 46} (1994), no. 2, 267–287.

 
%\bibitem[RM10]{RM10} M. Rumin, \emph{ Spectral density and Sobolev inequalities for pure and mixed states}, Geom. Funct. Anal. {\bf 20}(3) (2010), 817–844.

%\bibitem[S95]{S95} G. Schwarz, \emph{ Hodge decomposition—a method for solving boundary value problems}, Lecture Notes in Math., 1607, Springer-Verlag, Berlin, 1995. viii+155 pp.
   
 \bibitem[SC95]{SC95} C. Scott, \emph{ $L^p$ theory of differential forms on manifolds}, Trans. Amer. Math. Soc.   {\bf 347} (1995), no. {\bf 6}, 2075–2096.
 
 \bibitem[SC02]{SC02} Saloff-Coste, L. Aspects of Sobolev-Type Inequalities, Cambridge Univ. Press, Cambridge (2002).
 
%\bibitem[SY94]{SY94} R. Schoen and S. T. Yau, \emph{ Lectures on differential geometry,} Conf. Proc. Lecture Notes Geom. Topology, I International Press, Cambridge, MA, 1994, v+235 pp.

%\bibitem[W04]{W04} C. Wang, \emph{ The Calderón-Zygmund inequality on a compact Riemannian manifold}, Pacific J. Math.{\bf 217}(2004), no.{\bf 1}, 181–200.
 

\bibitem[WZ12]{WZ12} J. Wang and L. Zhou,\emph{ Gradient estimate for eigenforms of Hodge Laplacian}, Math. Res. Lett.   {\bf 19}(2012), no. {\bf 3}, 575–588.

\bibitem[Z94]{Z94} Y. Zhu, \emph{ A generalization of the Kodaira vanishing theorem and the Kodaira embedding theorem}, Thesis (Ph.D.)–Boston University
ProQuest LLC, Ann Arbor, MI, 1994, 69 pp.
  
\bibitem[Z11]{Z11} Q. Zhang, \emph{ Sobolev inequalities, heat kernels under Ricci flow and the Poincar\'e conjecture}, CRC Press, Boca Raton, FL, 2011.


\end{thebibliography}

\end{document}